\documentclass{scrartcl}

\usepackage{geometry}                		
\RedeclareSectionCommand[
  font=\normalfont\scshape\Large\centering
]{section}

\setkomafont{subsection}{\normalfont\bfseries\large}      
\setkomafont{subsubsection}{\normalfont\itshape\normalsize} 
\setkomafont{paragraph}{\normalfont\bfseries\normalsize}
\setkomafont{subparagraph}{\normalfont\normalsize\itshape}

\setkomafont{sectionentry}{\scshape}       
\setkomafont{pagenumber}{\normalfont}      
\setkomafont{title}{\scshape\Large\centering}

\renewenvironment{abstract}{
  \small
  \noindent\textsc{Abstract.}\ } 
{\par\vspace{1em}}

\usepackage{mathtools}
\usepackage{amssymb}
\usepackage{stmaryrd}
\usepackage{graphicx}
\usepackage{tikz}
\usepackage{tikz-cd} 

\newcommand{\Z}{\mathbb{Z}}
\newcommand{\RelAt}{\mathbf{Rel}^@}
\newcommand{\hooklongrightarrow}{\lhook\joinrel\longrightarrow}
\newcommand{\Prop}{\mathsf{Prop}}
\newcommand{\Type}{\mathsf{Type}}

\usepackage{amsmath}

\usepackage[english]{babel}
\usepackage{amsthm}

\usepackage{biblatex}
\usepackage[colorlinks=true,
            linkcolor=blue,      
            citecolor=teal,      
            urlcolor=magenta,    
            filecolor=purple     
]{hyperref}

\theoremstyle{definition}
\newtheorem{definition}{Definition}[section]
\newtheorem{example}{Example}[section]

\theoremstyle{plain}
\newtheorem{proposition}{Proposition}[section]

\theoremstyle{remark}

\newtheorem{corollary}{Corollary}[section]

\title{Type Theory for the Working Mathematical Music Theorist}
\author{Drew Flieder\thanks{\, Email: \texttt{drew@drewflieder.xyz}}}
\date{}

\begin{document}
\maketitle

\begin{abstract}
Many formal languages of contemporary mathematical music theory—particularly those employing category theory—are powerful but cumbersome: ideas that are conceptually simple frequently require expression through elaborate categorical constructions such as functor categories. This paper proposes a remedy in the form of a type-theoretic symbolic language that enables mathematical music theorists to build and reason about musical structures more intuitively, without relinquishing the rigor of their categorical foundations. Type theory provides a syntax in which elements, functions, and relations can be expressed in simple terms, while categorical semantics supplies their mathemusical interpretation. Within this system, reasoning itself becomes constructive: propositions and proofs are treated as objects, yielding a framework in which the formation of structures and the reasoning about them take place within the same mathematical language. The result is a concise and flexible formalism that restores conceptual transparency to mathemusical thought and supports new applications, illustrated here through the theory of voice-leading spaces.
\end{abstract}

\tableofcontents

\section{Introduction and Motivation}

Mathematical music theory increasingly relies on sophisticated formal tools to express and reason about musical structures. Yet this sophistication comes at a cost. Many constructions that are conceptually simple demand technically elaborate encodings, and intuitive reasoning about musical ideas often becomes entangled with highly technical mathematical machinery. The goal of this paper is to bridge this gap by developing a symbolic framework that is both \emph{intuitive} and \emph{rigorous}—one that allows mathemusicians to build and reason about musical structures in a simpler, skeletal formal language while retaining access to full semantic precision when desired.

Two closely related motivations guide this work:
\begin{enumerate}
\item To introduce a \emph{symbolic lingua franca} for mathemusicians: a language that makes the construction of mathemusical objects transparent and intuitive without sacrificing formal strength.
\item To provide a unified system in which the formation of musical structures and the reasoning about them are intrinsic to the same language.
\end{enumerate}

\subsection{A Symbolic Lingua Franca for Mathematical Music Theorists}
\label{sub:a-symbolic-lingua-franca-for-mathematical-music-theorists}

Many of the most expressive categorical approaches to mathematical music theory—beginning with the presheaf formalisms developed by Mazzola~\cite{mazzola2018topos,mazzola2002topos} and developed further in more recent work~\cite{flieder2024towards},  as well as functorial approaches such as~\cite{noll2005topos,popoff2018relational}—demonstrate how powerful such tools can be. Yet they also reveal a recurring tension: ideas that are simple to state in naïve mathematical terms become highly encumbered when expressed in functorial language. For instance, in the category $\mathbf{Set}$, to ``choose an element'' $x \in X$ is simply to specify a map $x : 1 \to X$ from the singleton set. In a presheaf category, however, the analogous notion involves a natural transformation from the terminal presheaf $\mathbf{1}$ to a presheaf $\hat{X}$, with components indexed by every object and compatibility conditions imposed by naturality.

This abstraction is in many ways inevitable, since generality and structural precision often require more elaborate mediating machinery. Yet it can obscure the intuitive simplicity of the underlying ideas. In my own work, I have frequently spoken of an ``element'' of a presheaf as though it were a set-theoretic element, while strictly speaking it is a family of maps subject to coherence conditions (i.e., a natural transformation). The technical apparatus remains ``under the hood,'' but it is not always relevant to the conceptual work at hand.

Here is where type theory proves valuable. It provides a formal syntax that allows us to reason in familiar terms—types, elements, functions, and relations—without continually invoking the underlying categorical machinery. Much as one need not understand automotive engineering to drive a car, a mathemusician need not work directly with functors and natural transformations to reason about the ``element'' of a presheaf $X$ representing some mathemusical structure. Moreover, since syntax is independent of semantics, one can specify the \emph{form} of a construction before committing to any particular interpretation—a principle we will illustrate in several examples throughout this text. Different semantic models can then yield distinct musical meanings, each grounded in a shared symbolic core.\footnote{See Example~\ref{ex:interpretation-of-pitch-classes-and-intervals}, where we present the theory of a \emph{generalized interval system} following \cite{lewin2007generalized}.} In this way, type theory functions not merely as a lighter formalism, but as a truly unifying language, capable of mediating among diverse mathemusical contexts within a single expressive framework.

\subsection{Reasoning as Constructive}

A further virtue of type theory is that it unifies the constructive and logical dimensions of mathematical reasoning. Traditional approaches often treat these as separate layers: a formal system such as ZFC builds mathematics on top of first-order logic, while frameworks like Mazzola’s similarly distinguish the languages of forms, denotators, and predicates as occupying distinct ontological strata.\footnote{For Mazzola’s account of this layered architecture, see \cite[Section~1.2]{mazzola1997semiotics} and \cite[Chapters~6 and~18]{mazzola2002topos}.} Type theory, by contrast, integrates these roles. It brings together what might be called the \emph{constructive} side of mathematics (the building of objects) and the \emph{verificative} side (the formation of predicates that apply to those objects to yield propositions) within a single formal language. Moreover, verification itself becomes constructive: propositions and proofs are not merely external judgments about mathematical entities but new objects in their own right (see Appendix~\ref{sec:propositions-as-types} on the \emph{propositions-as-types} paradigm). This unification deepens the expressive capacity of mathemusical discourse by making reasoning itself intrinsically constructive.\footnote{Examples~\ref{ex:proof-all-interval} and \ref{ex:proof-domfunc-leading-tone} demonstrate the constructive nature of propositions, such as ``$x$ (a pitch-class set) is all-interval'' and ``The dominant of any key contains the leading tone of that key.''}

To illustrate the point concretely, consider a basic tonal fact: in harmonic minor, the V chord functions as a dominant, whereas in natural minor it does not, since without the leading tone there is no strong dominant tendency. The sole musical difference here is the leading tone, which has decisive implications for the inferential behavior of the harmony. In type-theoretic terms, this corresponds to the fact that certain propositions become derivable in the context where the scale contains a leading tone, and fail to be derivable when it does not. In a harmonic-minor context, we can derive that V is dominant; in a natural-minor context, we \emph{cannot}. The derivation of dominance is thus sensitive to the tonal context. 

This formal phenomenon mirrors how musicians actually think. Suppose one is composing a passage in minor and many of the melodic tones lie in natural minor. One knows that a dominant chord is not available in that context. To form an authentic cadence, one shifts to harmonic minor precisely in order to make available a dominant.

In the type-theoretic setting, this is literally expressible: one changes the context, and thereby changes which propositions are derivable. A proposition that did not hold (``V is dominant'') now becomes derivable. This is not a meta-linguistic explanation. It is a formal transition internal to the system.

\subsection{Structure of the Paper}
\label{sub:paper-structure}

The remainder of this paper is organized as follows. 

\begin{itemize}
    \item Section~\ref{sec:the-type-theoretic-machinery} introduces the technical foundations of type theory, including the notions of signatures, typing judgments, and classifying categories. To maintain contact with the intended music-theoretical aims, we intersperse these developments with elementary examples of typing judgments in musical contexts. We then introduce the principal type constructors, which enable the recursive synthesis of new types, and show how these give rise to a higher-order logic once a signature and a family of constructors are fixed. The section concludes with a brief discussion of the internal language of a category.
    
    \item Section~\ref{sec:categorical-semantics} presents the notion of categorical semantics, showing how the purely syntactic machinery of type theory may be interpreted in semantically rich environments. This interpretation endows syntactic constructions with mathemusical meaning, thereby linking the formal apparatus to conceptual content. The section closes with several illustrative mathemusical examples.
    
    \item Section~\ref{sec:an-application-to-the-theory-of-voice-leading-spaces} presents a more developed application of the framework to the theory of voice-leading spaces. While the details of this section presuppose the material of the preceding two, readers primarily interested in musical applications may wish to skim it first to gain a sense of the kinds of structures that become expressible once the theoretical foundations are in place.
    
    \item Section~\ref{sec:conclusion} offers concluding remarks and identifies directions for future work.
\end{itemize}

We also provide three appendices—\ref{sec:categorical-semantics-a-deeper-look}, \ref{sec:propositions-as-types}, and \ref{sec:w-types}—which furnish additional technical material: respectively, a deeper treatment of categorical semantics, an account of propositions as types, and an introduction to inductive type constructors through $\mathsf{W}$-types.

\section{The Type-Theoretic Machinery}
\label{sec:the-type-theoretic-machinery}

In this section we present the foundational components of type theory that will serve as the formal backbone of our framework. At first glance, type theory looks much like a naïve set theory: in place of sets, we work with \emph{types}, which classify terms much as sets classify elements. We also have \emph{functions} (mappings between types) and \emph{relations} (logical conditions that may hold between their terms). 

\subsection{Signatures}

To begin, we define what is called a \emph{signature} \cite[Section D.1]{johnstone2002sketches}: 

\begin{definition}[Signature]
A \emph{signature} $\Sigma$ is a triple 
\[ 
\Sigma = (\Sigma\normalfont{\text{-Type}}, \Sigma\normalfont{\text{-Fun}}, \Sigma\normalfont{\text{-Rel}}),\]
 where:
\begin{enumerate}
    \item $\Sigma\normalfont{\text{-Type}}$ is a set of \emph{types}.
    \item $\Sigma\normalfont{\text{-Fun}}$ is a set of \emph{function symbols}, written as
    \[ f : A_1, \ldots, A_n \longrightarrow B, \]
    where $A_1, \ldots, A_n$ is a list of types and $B$ is a type. The \emph{type} of $f$ is the list $A_1, \ldots, A_n, B$. The number $n$ is called the \emph{arity} of the function, and if $n = 0$, then $f$ is called a \emph{constant symbol}. 
    \item $\Sigma\normalfont{\text{-Rel}}$ is a set of \emph{relation symbols}, written as
    \[ R \hooklongrightarrow A_1, \ldots, A_n. \]
    The \emph{type} of $R$ is $A_1, \ldots, A_n$. The number $n$ is called the \emph{arity} of the relation, and if $n = 0$, then $R$ is called an \emph{atomic proposition}. The symbol $\hookrightarrow$ indicates that $R$ is a relation on the given types, with $R$ corresponding to a subobject of the Cartesian product $A_1 \times \cdots \times A_n$. This notation emphasizes the role of relations as subobjects of product objects in a category.
\end{enumerate}
\end{definition}

A signature therefore specifies the basic syntactic structure of a type-theoretic system. It specifies the basic types, function symbols, and relation symbols. The elements of $\Sigma\text{-Type}$ correspond to distinct kinds of objects, while the elements of $\Sigma\text{-Fun}$ and $\Sigma\text{-Rel}$ specify how objects of these types relate to one another through functions and relations, respectively. Additionally, we assume a set $V = \{ x, y, \ldots \}$ of variables, which will be utilized in the construction of terms and expressions in the system.

As we will see, the types we introduce here have a close correspondence with objects in a category, while function symbols correspond to morphisms between these objects. This categorical perspective will provide a bridge between the symbolic worlds of logic and the objective worlds of mathematical structure.

\begin{example}[Signature of a group]
\label{ex:signature-of-group}
In this example we show how the concept of a mathematical group can be expressed purely syntactically as a signature with associated equational identities.

A \emph{group} is a set $G$ equipped with a binary operation $\star : G \times G \to G$ satisfying the following axioms:
\begin{enumerate}
\item \textbf{Associativity:} For every $a,b,c \in G$,
\[
(a \star b) \star c = a \star (b \star c)
\]
\item \textbf{Identity element:} There exists a unique $e \in G$ such that for all $g \in G$, 
\[
g \star e = e \star g = g.
\]
\item \textbf{Inverses:} For every $g \in G$, there exists a unique $h \in G$ such that $g \star h = e$.
\end{enumerate}

We now express this structure at the level of a \emph{signature}:
\[
\mathbb{G} = (\mathbb{G}\text{-Type}, \mathbb{G}\text{-Fun}, \mathbb{G}\text{-Rel}),
\]
following \cite[Section~4.1]{flieder2024towards} and \cite[Chapter~4]{awodey2010category}. The signature consists of a single type $G$ and no relation symbols. It includes the function symbols
\[
\star : G \times G \to G, \quad e : \mathbf{1} \to G, \quad \mathrm{inv} : G \to G,
\]
corresponding respectively to the binary operation, the identity element, and the inverse operation.

A \emph{model} (see Section~\ref{sec:categorical-semantics}) of this signature is any interpretation of these function symbols in a category with finite products that satisfies the following equational identities, corresponding to the three axioms above. These equations are most naturally expressed as the commutativity of the following diagrams.

\begin{enumerate}
\item \textbf{Associativity.} The associativity of $\star$ is captured by the commutativity of:
\[
\begin{tikzcd}
	{(G \times G) \times G} && {G \times (G \times G)} \\
	\\
	{G \times G} && {G \times G} \\
	& G
	\arrow["\cong", from=1-1, to=1-3]
	\arrow["{\langle \star, \ \mathrm{id}_G\rangle}"', from=1-1, to=3-1]
	\arrow["{\langle \mathrm{id}_G, \ \star\rangle}", from=1-3, to=3-3]
	\arrow["\star"', from=3-1, to=4-2]
	\arrow["\star", from=3-3, to=4-2]
\end{tikzcd}
\]

\item \textbf{Identity.} That $e : \mathbf{1} \to G$ picks out the identity element is ensured by:
\[
\begin{tikzcd}
	G && {G \times G} \\
	\\
	{G \times G} && G
	\arrow["{\langle e, \ \mathrm{id}_G \rangle}", from=1-1, to=1-3]
	\arrow["{\langle \mathrm{id}_G, \ e \rangle}"', from=1-1, to=3-1]
	\arrow["\star", from=1-3, to=3-3]
	\arrow["\star"', from=3-1, to=3-3]
\end{tikzcd}
\]
Here we write $e$ for the composite $e! : G \xrightarrow{!} 1 \xrightarrow{e} G$.

\item \textbf{Inverses.} That $\mathrm{inv} : G \to G$ gives inverses with respect to $\star$ is ensured by:
\[\begin{tikzcd}
	{G \times G} && G && {G \times G} \\
	\\
	{G \times G} && G && {G \times G}
	\arrow["{\langle  \mathrm{id}_G, \; \mathrm{inv} \rangle}"', from=1-1, to=3-1]
	\arrow["{\langle \mathrm{id}_G, \; \mathrm{id}_G \rangle}"', from=1-3, to=1-1]
	\arrow["{\langle \mathrm{id}_G, \; \mathrm{id}_G \rangle}", from=1-3, to=1-5]
	\arrow["e", from=1-3, to=3-3]
	\arrow["{\langle \mathrm{inv}, \; \mathrm{id}_G \rangle}", from=1-5, to=3-5]
	\arrow["\star"', from=3-1, to=3-3]
	\arrow["\star", from=3-5, to=3-3]
\end{tikzcd}\]
\end{enumerate}

Thus the \emph{theory of groups} is presented by the signature $\mathbb{G}$ together with these equational identities. Later (see Example \ref{ex:interpretation-of-pitch-classes-and-intervals}), we will see how this group signature serves as one of the ingredients in Lewin’s definition of a \emph{generalized interval system} \cite{lewin2007generalized}.
\end{example}

\subsection{Terms, Contexts, and Typing Judgments}
\label{sub:terms-contexts-and-typing-judgments}

In mathematical music theory, we often wish not merely to refer to mathematical structures—such as scales, chords, or voice-leading spaces—but to specify the \emph{conditions} under which their elements may be constructed. For example, if we have a structure $\mathsf{IVLS}$ representing musical intervals, we may ask under what context an interval can be formed. Given two pitch classes $x$ and $y$, we can construct an interval by taking their difference $y - x$, so we can understand the terms of $\mathsf{IVLS}$ as depending on pairs of pitch classes. In general, the construction of a term of a certain type depends on the availability of terms from other types. We therefore introduce the notions of \emph{terms}, \emph{contexts}, and \emph{typing judgments}, which provide the basic machinery for forming \emph{judgments} in type theory.\footnote{See \cite[Chapter~2]{jacobs1999categorical} for a more detailed treatment.}

For a signature $\Sigma$, the types it specifies correspond to \emph{contexts}, while its function symbols correspond to \emph{terms}, as we will soon see.

\subsubsection{Terms}
\label{subsub:terms}

We begin with the notion of a \emph{term}. For a type $A$, the expression $x : A$ asserts that $x$ is a \emph{term of type} $A$. This plays a role analogous to membership in set theory: just as $x \in A$ states that $x$ is an element of the set $A$, so $x : A$ states that $x$ is a term of the type $A$. However, because a signature specifies only types, function symbols, and relation symbols—not ``elements'' in the set-theoretic sense—we require a systematic way to speak about terms constructed from these symbols. As we will see, terms are closely associated with function symbols, though this relationship will be made precise as our development proceeds. 

\subsubsection{Contexts}
\label{subsub:contexts}

In order to reason effectively about how terms depend on one another, we must keep track of the assumptions under which they are constructed. This is the role of a \emph{context}. A context is an ordered sequence of variable declarations,
\[
\Gamma = x_1 : A_1, \ldots, x_n : A_n,
\]
where each $x_i$ is a variable of type $A_i$. Intuitively, a context specifies the assumptions under which a judgment is made, recording the variables currently in scope and their associated types. 

Contexts may be concatenated: given $\Gamma$ as above and another context 
\[
\Delta = y_1 : B_1, \ldots, y_m : B_m,
\]
their concatenation is written
\[
\Gamma, \Delta = x_1 : A_1, \ldots, x_n : A_n, y_1 : B_1, \ldots, y_m : B_m.
\]

There are also situations in which we may wish to substitute terms for the variables in a context. In a context $\overrightarrow{x} = x_1 : A_1, \ldots, x_n : A_n$, we may wish to substitute terms $a_i : A_i$ for the variables $x_i : A_i$. If $\overrightarrow{a} = a_1, \ldots, a_n$ is a list of terms of the same length and type as $\overrightarrow{x}$, we write $[\overrightarrow{a} / \overrightarrow{x}]$ for the operation that substitutes $a_i$ for $x_i$ for each $i \leq n$. This process allows us to \emph{instantiate} the variables within a term with concrete terms.

\subsubsection{Typing Judgments}
\label{subsub:typing-judgments}

With contexts in place, we can now form \emph{typing judgments}. A typing judgment is an expression of the form
\[
\Gamma \;\vdash\; a : A
\]
which states that, relative to the assumptions recorded in the context $\Gamma$, the term $a$ has type $A$. In other words, it asserts that a term has a given type, always relative to a context $\Gamma$.

For example, in the empty context $\Gamma = ()$, a typing judgment
\[
() \;\vdash\; a : A
\]
indicates that $a$ is a constant term of type $A$. Such a constant symbol corresponds to a function of arity 0, analogous to a set function $1 \to A$ where $1$ is a singleton set. This perspective allows us to view constants as ``global'' elements of $A$—terms that require no assumptions to construct.

By contrast, using our interval type $\mathsf{IVLS}$ introduced above, we can form an interval $i : \mathsf{IVLS}$ given pitch classes $x$ and $y$ of some pitch type $\mathsf{Pitch}$. This dependency is expressed by the typing judgment
\[
x : \mathsf{Pitch}, \; y : \mathsf{Pitch} \;\vdash\; i : \mathsf{IVLS}.
\]
Here the context records the assumptions required to construct an interval: namely, the availability of two pitch classes from which the interval may be formed (see Example~\ref{ex:constructing-intervals} for a more developed account).

\subsection{The Classifying Category of a Signature}
\label{sub:the-classifying-category-of-a-signature}

With the machinery of terms, contexts, and typing judgments in place, we now turn to the morphisms they induce, which will allow us to define what is called the \emph{classifying category} of a signature. 

Given a context 
\[
\Gamma = x_1 : A_1, \ldots, x_n : A_n
\]
and a typing judgment 
\[
\Gamma \;\vdash\; t_1 : T_1, \ldots, t_m : T_m,
\]
such a typing judgment corresponds to a morphism
\[
A_1, \ldots, A_n \xrightarrow{(t_1, \ldots, t_m)} T_1, \ldots, T_m.
\]
This arises by viewing the tuple $(t_1, \ldots, t_m)$ as a map from the product $A_1 \times \cdots \times A_n$ to the product $T_1 \times \cdots \times T_m$. Each component $t_i$ defines a morphism
\[
t_i : A_1 \times \cdots \times A_n \longrightarrow T_i,
\]
and together they map each tuple $\overrightarrow{a} = (a_1, \ldots, a_n)$ from the product $A_1 \times \cdots \times A_n$ to the tuple
\[
(t_1(\overrightarrow{a}), \ldots, t_m(\overrightarrow{a}))
\]
in $T_1 \times \cdots \times T_m$. 

This construction forms the foundation for the notion of a \emph{classifying category}.

\begin{definition}[Classifying Category]
For a signature $\Sigma$, its \emph{classifying category} $\mathcal{C}\ell(\Sigma)$ is defined as follows: its objects are contexts, and its morphisms $\Gamma \to \Delta$, where $\Delta = y_1 : B_1, \ldots, y_m : B_m$, are $m$-tuples of terms $(t_1, \ldots, t_m)$ such that $\Gamma \;\vdash\; t_i : B_i$ for each $i \leq m$.\footnote{See \cite[124]{jacobs1999categorical}.}
\end{definition}
 
\subsection{Musical Examples of Typing Judgments}
\label{sub:musical-examples-of-typing-judgments}

We now present musical examples of typing judgments.

\subsubsection{Typing Judgments for Data Types}
\label{subsub:typing-judgments-for-data-types}

In what follows, we display typing judgments which construct terms of basic musical data types.

\begin{example}[Constructing intervals]  
\label{ex:constructing-intervals}
Earlier (Section \ref{subsub:typing-judgments}) we introduced the typing judgment 
\[
x : \mathsf{Pitch}, \; y : \mathsf{Pitch} \; \vdash \; i : \mathsf{IVLS}
\]
which expresses that an interval can be formed from two terms of a pitch type $\mathsf{Pitch}$.

In the classifying category, this interval is given by a morphism  
\[
\begin{matrix}
i : & \mathsf{Pitch}, \mathsf{Pitch} & \longrightarrow & \mathsf{IVLS} \\
& (x, y) & \longmapsto & y - x,
\end{matrix}
\]  
which assigns to each pair of pitch classes their difference. Thus, the context provides the \emph{input data} (two pitch classes), and the typing judgment asserts that from this data one can construct a corresponding interval $i(x, y)$.

By contrast, we may also form a typing judgment in the \emph{empty context}, e.g.  
\[
() \;\vdash\; i : \mathsf{IVLS},
\]  
which simply specifies an interval independently of any assumptions. Categorically, this  corresponds to a \emph{global element} $i : () \to \mathsf{IVLS}$, where $()$ is the terminal object.  
\end{example}

Already we begin to see what the type-theoretic system affords. It allows us to start constructing a language without yet committing to any particular kind of object that will interpret our constructions. In the preceding example, we did not specify what object should serve as the interpretation of the abstract type~$\mathsf{Pitch}$; we only know that it will eventually stand for pitches of some kind. At this stage, however, we have not decided whether $\mathsf{Pitch}$ is to be realized as a set in~$\mathbf{Set}$, as a module presheaf in Mazzola’s category $\mathbf{Mod}^@$, or as a presheaf in $\mathbf{Rel}^@$, to name a few possibilities. Nor have we fixed the underlying universe of pitches: $\mathsf{Pitch}$ might be interpreted as the universe $\mathbb{Z}$ of equal-tempered pitches, where $0$ denotes middle-C, $-1$ middle-B, and so on; or as the twelve-tone universe $\mathbb{Z}_{12}$; or as some other collection.

What we \emph{have} provided, at this point, is a rule for constructing terms of the type of pitch intervals, independently of how the abstract types are ultimately interpreted. Given two pitches (or pitch classes) $x$ and $y$, we may form their interval $y - x$ (supposing subtraction is well-defined).

Let us consider another example.

\begin{example}[Building chords]
Suppose we have a type $\mathsf{Chord}$ of chords. We can define a rule for constructing a term $c : \mathsf{Chord}$ as follows. Let $\mathsf{DiatonicScale}$ be the type of major and minor diatonic scales, and $\mathsf{ScaleDegree}$ the type of scale degrees (1 through 7). Given a scale $s : \mathsf{DiatonicScale}$ and a degree $d : \mathsf{ScaleDegree}$, we may construct the triad in the scale $s$ rooted at degree $d$:
\[
s : \mathsf{DiatonicScale}, \; d : \mathsf{ScaleDegree} \;\vdash\; \mathsf{triad}(s,d) : \mathsf{Chord}.
\]

Once again this corresponds to a morphism in the classifying category,
\[
\begin{matrix}
\mathsf{triad} : & \mathsf{DiatonicScale}, \mathsf{ScaleDegree} & \longrightarrow & \mathsf{Chord} \\
& (s, d) & \longmapsto & \{\, s_d,\; s_{d+2},\; s_{d+4} \,\},
\end{matrix}
\]
which sends a scale and a root degree to the chord formed by stacking every other note beginning at $s_d$.
\end{example}

\subsubsection{Typing Judgments for Propositions}
\label{subsub:typing-judgments-for-propositions}

We now turn to the role of relation symbols and propositions. For our purposes, we assume that a signature $\Sigma$ contains a distinguished type $\mathsf{Prop}$, representing propositions (i.e.\ truth values). This allows us to regard relation symbols as functions into $\mathsf{Prop}$. Specifically, any relation symbol 
\[
R \hooklongrightarrow A_1, \ldots, A_n
\]
can equivalently be represented as a function symbol 
\[
R : A_1, \ldots, A_n \longrightarrow \mathsf{Prop}.
\]
Thus terms of type $\mathsf{Prop}$ are to be understood as \emph{propositional functions}. A proposition $R$ in context $x_1 : A_1, \ldots, x_n : A_n$ then corresponds to a typing judgment
\[
x_1 : A_1, \ldots, x_n : A_n \; \vdash \; R(x_1, \ldots, x_n) : \mathsf{Prop},
\]
that is, a rule for constructing a proposition given the assumptions collected in the context.

This perspective highlights the equivalence between monomorphisms into $A_1 \times \cdots \times A_n$ and propositional functions on $A_1 \times \cdots \times A_n$. Morphisms $T \to \mathsf{Prop}$ are in one-to-one correspondence with relations on $T$. For example, if $\mathsf{Prop}$ contains the two truth values $\top$ and $\bot$, then for a relation $P$ to be true of $t : T$ is to have $P(t) = \top$, and the collection of all $t$ for which $P(t) = \top$ is precisely the relation $P \hookrightarrow T$.

\begin{example}[Scale membership]
\label{ex:scale-membership}
Let $\mathsf{PC}$ be the type of pitch classes and $\mathsf{Scale}$ the type of scales. The relation ``$p$ belongs to $s$'' can be formalized as a propositional function
\[
\in \; \colon \mathsf{PC}, \mathsf{Scale} \longrightarrow \mathsf{Prop}.
\]
In the context $p : \mathsf{PC},\; s : \mathsf{Scale}$ we obtain the typing judgment
\[
p : \mathsf{PC},\; s : \mathsf{Scale} \;\vdash\; (p \in s) : \mathsf{Prop},
\]
which expresses that ``$p$ is a member of $s$'' is a well-formed proposition relative to the given context.
\end{example}

\begin{example}[Functional harmony]
\label{ex:functional-harmony}
Let $\mathsf{Chord}$ be the type of chords and $\mathsf{Key}$ the type of keys. The statement ``$c$ is the dominant in $k$'' can be formalized as a propositional function
\[
\mathsf{dom} : \mathsf{Chord}, \mathsf{Key} \longrightarrow \mathsf{Prop}.
\]
Thus, in the context $c : \mathsf{Chord},\; k : \mathsf{Key}$ we obtain the typing judgment
\[
c : \mathsf{Chord},\; k : \mathsf{Key} \;\vdash\; \mathsf{dom}(c,k) : \mathsf{Prop},
\]
expressing that ``$c$ is the dominant in $k$'' is a proposition whose truth depends on the chord and the key supplied by the context.
\end{example}

\begin{example}[Dominance depends on context]
\label{ex:dominance-context-dependent}

Let $\mathsf{NoteName}$ be the type of note names (e.g.\ C$\sharp$), and let $\mathsf{sctype} : \mathsf{NoteName} \to \mathsf{Scale}$ assign to each note name a seven-note scale of a given kind (e.g.\ harmonic minor, natural minor).

Fix a context
\[
\Gamma \;=\; n:\mathsf{NoteName},\; \mathsf{sctype}(n) : \mathsf{Scale},
\]
and assume that each scale $\mathsf{sctype}(n)$ carries a scale-degree function
\[
\mathsf{scdeg}_{\mathsf{sctype}(n)} : \mathsf{sctype}(n) \longrightarrow \{ \hat{1}, \hat{2}, \ldots, \hat{7} \}.
\]

We define, in this context, the proposition
\[
\Gamma \;\vdash\; \mathsf{containsLeadingTone}(\mathsf{sctype}(n)) : \mathsf{Prop}
\]
to hold when the pitch of scale degree $\hat{7}$ lies a semitone below the pitch of scale degree $\hat{1}$.

Now consider two specifications of $\Gamma$:
\[
\Gamma_{\mathsf{harm}} \;=\; n:\mathsf{NoteName},\;  \mathsf{sctype}(n) := \mathsf{harm}(n)
\]
and
\[
\Gamma_{\mathsf{nat}} \;=\; n:\mathsf{NoteName},\;  \mathsf{sctype}(n) := \mathsf{nat}(n)
\]
where $\mathsf{harm}(n)$ and $\mathsf{nat}(n)$ denote the harmonic resp.\ natural minor built on $n$.

We state:
\[
\mathsf{dominant}(V(\mathsf{sctype}(n))) \;\iff \; \mathsf{containsLeadingTone}(\mathsf{sctype}(n)).
\]
This expresses the musical fact that the $V$ chord of $\mathsf{sctype}(n)$ is dominant
if and only if that scale contains the leading tone.

Then in harmonic minor we have a derivation---viz.\ for any substitution of the terms in the context $\Gamma_\mathsf{harm}$, the proposition asserting the dominance of $V$ is derivable:
\[
\Gamma_{\mathsf{harm}} \;\vdash\; \mathsf{dominant}(V(\mathsf{sctype}(n))) = \mathsf{true}.
\]
By contrast, in natural minor we do not:
\[
\Gamma_{\mathsf{nat}} \;\not\vdash\; \mathsf{dominant}(V(\mathsf{sctype}(n))) =\mathsf{true}.
\]

Thus the truth of the proposition ``$V$ is dominant'' is context-dependent: it is derivable in $\Gamma_\mathsf{harm}$ and not derivable in $\Gamma_\mathsf{nat}$.
\end{example}

As the preceding examples make clear, reasoning in type theory always occurs within a specific context $\Gamma$. The context defines the types—the ``universe of discourse''—over which propositions are formed. For instance, in the context $c : \mathsf{Chord},\; k : \mathsf{Key}$, the proposition that ``$c$ is the dominant of $k$'' is well-formed, since the types of $c$ and $k$ are specified. By contrast, the same proposition without a context, yielding
\[
() \;\vdash\; \mathsf{dom}(c,k) : \mathsf{Prop},
\]
is ill-formed. Thus, when constructing compound propositions using logical connectives, the context $\Gamma$ must always supply the types of all free variables on which the propositional function depends.

\subsection{Type Constructors}
\label{sub:type-constructors}

Thus far we have introduced signatures, explained how typing judgments are generated from them, shown how the classifying category of a signature emerges from this machinery, and illustrated these ideas with preliminary mathemusical examples.

Yet a proposal for a lingua franca of mathematical concept-building would remain incomplete without the capacity to express the universal constructions that pervade mathematics. These include, for instance, products, coproducts, function types, and subobject types. Introducing such constructions allows us to extend a signature by enabling the recursive synthesis of new types through \emph{type constructors}.

In type theory, such constructions are specified by \emph{formation rules}. These are written as inference rules with a horizontal bar: the judgments above the bar are the \emph{premises}, and the judgment below the bar is the \emph{conclusion}. 

From now on, we abbreviate typing judgments of the form
\[
() \;\vdash\; a : A
\] 
in the empty context as simply $a : A$.

Before introducing the type constructors, we first note two particularly simple types:\begin{itemize}
\item The \emph{null type} (or empty type), usually written $\mathbf{0}$, contains no terms.  
\item The \emph{unit type}, written $\mathbf{1}$, contains exactly one canonical term $* : \mathbf{1}$.  In logical terms, these correspond respectively to falsity and truth, and in set-theoretic terms to the empty set and a singleton set.
\end{itemize}

We begin with the familiar cases of product, coproduct, and function types, encouraging the reader to think of them in terms of sets (Cartesian products, disjoint unions, and hom-sets). We then proceed to dependent products and dependent sums, which are less familiar. For reasons of space and complexity, we defer the presentation of inductive types to Appendix \ref{sec:w-types}.

Type theory distinguishes among:
\begin{itemize}
\item \emph{Type formation rules}, which specify how to form a type from other types.
\item \emph{Term introduction rules}, which specify how to construct a term of a given type.
\item \emph{Term elimination rules}, which specify how to use a term of a given type.
\end{itemize}

In what follows our emphasis will be on formation rules. Introduction and elimination rules will sometimes be omitted, as they can become technically involved and risk obscuring our primary motivation. (For a fuller treatment, see, for instance, \cite[Chapter 1 and Appendix A]{hott2013}.) To illustrate the relevance of each type constructor to musical thought, we will also provide a brief mathemusical example of each.

\subsubsection{Product Types}

Given types $A$ and $B$, we can form their product type:
\[
\frac{A : \Type \quad B : \Type}{A \times B : \Type}
\]
Here, the premise judgments are $A : \Type$ and $B : \Type$, and the conclusion is the judgment $A \times B : \Type$.  The term introduction rule for products is:
\[
\frac{a : A \quad b : B}{(a, b) : A \times B}
\]
The corresponding elimination rules are the projections:
\[
\frac{p : A \times B}{\mathsf{pr}_1(p) : A}
\qquad
\frac{p : A \times B}{\mathsf{pr}_2(p) : B}
\]

\begin{example}[MIDI note]
Suppose we wish to define a type $\mathsf{MidiNote}$ that encapsulates the essential data of a MIDI note: its onset, pitch, velocity, and duration. Given types $\mathsf{Onset}$, $\mathsf{Pitch}$, $\mathsf{Velocity}$, and $\mathsf{Duration}$ corresponding to each of these components, we can form the product type
\[
\mathsf{MidiNote} \coloneqq \mathsf{Onset} \times \mathsf{Pitch} \times \mathsf{Velocity} \times \mathsf{Duration}.
\]
A term $(o, p, v, d) : \mathsf{MidiNote}$ then consists of an onset, a pitch, a velocity, and a duration.\end{example}

\subsubsection{Coproduct Types}

Given types $A$ and $B$, we can form their coproduct type:
\[
\frac{A : \Type \quad B : \Type}{A + B : \Type}
\]
The term introduction rules for coproducts are:
\[
\frac{a : A}{\mathsf{inl}(a) : A + B} 
\qquad 
\frac{b : B}{\mathsf{inr}(b) : A + B}
\]
These state that given $a : A$ (respectively $b : B$), we may construct a term of $A + B$ by applying the canonical left (resp. right) injection.

\begin{example}[Pitch or rest]
Suppose we have a type $\mathsf{Pitch}$ and a singleton type $\mathsf{Rest}$ whose unique term denotes a rest. We can then define the coproduct
\[
\mathsf{Event} \coloneqq \mathsf{Pitch} + \mathsf{Rest}.
\]
This type may be used, for example, to form an $n$-fold product representing an $n$-voice musical texture. Each harmonic event can then be conceived as a term
\[
(x_1, \ldots, x_n) : \mathsf{Event}^n
\]
in which some voices produce pitches while others rest.
\end{example}

\subsubsection{Function Types}

Given types $A$ and $B$, we can form the type of functions from $A$ to $B$:
\[
\frac{A : \Type \quad B : \Type}{(A \to B) : \Type}
\]
The corresponding elimination rule is application:
\[
\frac{f : A \to B \quad a : A}{f(a) : B}
\]
That is, given a function $f : A \to B$ and an argument $a : A$, we may apply $f$ to $a$ to obtain a term of type $B$.

The special case of a function type with codomain $\Prop$ illustrates the connection with logic.  If $p : X \to \Prop$, then $p$ is a propositional function over $X$, and term elimination corresponds to applying the propositional function to a subject:  
\[
\frac{p : X \to \Prop \quad x : X}{p(x) : \Prop}
\]
In other words, given a propositional function $p$ on $X$ and a term $x : X$, application yields the proposition $p(x)$.

\begin{example}[Sequences as functions]
Suppose we have a type $[n]$ representing the natural numbers from $1$ to $n$ inclusive, and another type $\mathsf{Param}$ representing some musical parameter (e.g., pitch, duration, velocity, etc.). The function type
\[
\mathsf{Seq}(n, \mathsf{Param}) \coloneqq [n] \to \mathsf{Param}
\]
then represents the type of $n$-length sequences whose elements are drawn from $\mathsf{Param}$.
\end{example}

\subsubsection{Dependent Product Types}

Dependent types, as the name suggests, allow the formation of types that vary with terms of another type. The first kind we introduce is the \emph{dependent product type}, also called the \emph{dependent function type} or simply the $\Pi$-\emph{type}.

Suppose we have $B : A \to \Type$, which in set-theoretic terms corresponds to an indexed family of sets, assigning to each $a : A$ a type $B(a)$. The $\Pi$-type is then the Cartesian product
\[
\prod_{x : A} B(x)
\]
of this indexed family of types. 

It is called a dependent function type because it represents functions whose codomain varies with the input. The ordinary function type is the special case where $B$ is constant. 

The formation rule for the dependent product type is:
\[
\frac{A : \Type \quad x : A \;\vdash\; B : \Type}{\prod_{x : A} B(x) : \Type}
\]

The elimination rule is application:
\[
\frac{f : \prod_{x : A} B(x) \quad a : A}{f(a) : B(a)}.
\]

\begin{example}[Tuning pitch classes]
The tuning theory developed in \cite{flieder2025lifting} provides a natural example of dependent function types. There, the basis of the theory is a group homomorphism
\[
\varphi : \mathbb{Z}_{n} \longrightarrow \mathbb{Z}_{m}
\]
defined by $\varphi(x) = kx$, where $k = m/n$. This map represents a perfectly symmetric embedding of the $n$-note pitch-class universe into the $m$-tone equally tempered scale.

To model ``distortions'' of this symmetry—that is, tunings that preserve the general structure but allow individual pitch classes to deviate slightly—we allow each pitch class $x : \mathbb{Z}_n$ to be tuned to any value within the range
\[
[\varphi(x), \varphi(x) + k - 1] \subset \mathbb{Z}_m.
\]
This situation is naturally captured by a dependent function type, since the codomain associated with each term $x : \mathbb{Z}_n$ depends on $x$ itself. We can express this dependency with the judgment
\[
x : \mathbb{Z}_n \;\vdash\; \delta^* : \mathsf{Type}
\]
where $\delta^*(x) = [\varphi(x), \varphi(x) + k - 1]$. This gives rise to the dependent function type
\[
\prod_{x : \mathbb{Z}_n} \delta^*(x).
\]
A term $t : \prod_{x : \mathbb{Z}_n} \delta^*(x)$ is then a specific tuning of $\mathbb{Z}_n$ within $\mathbb{Z}_m$, assigning to each pitch class $x$ a concrete tuned value within its allowable range.
\end{example}

\subsubsection{Dependent Sum Types}

A \emph{dependent sum type}, \emph{dependent pair type}, or simply $\Sigma$-\emph{type} forms pairs where the second component depends on the value of the first. It corresponds to the set-theoretic coproduct of sets indexed by a given set.\footnote{See \cite[30]{hott2013}.}

Specifically, given an indexed family $B : A \to \Type$ of types, the $\Sigma$-type is the indexed sum
\[
\sum_{x : A} B(x)
\]
but which we may think of as the type of pairs $(x, b)$ where $b : B(x)$. 

The formation rule for the dependent sum type is:
\[
\frac{A : \Type \quad x : A \;\vdash\; B : \Type}{\sum_{x : A} B(x) : \Type}
\]

\begin{example}[Transporters]
\label{ex:transporters}
Suppose we have a type $\mathsf{PC}$ of pitch classes, a type $S$, and a left $S$-action on $\mathsf{PC}$, given by a function
\[
\sigma : S \times \mathsf{PC} \longrightarrow \mathsf{PC}.
\]

We can define a dependent type
\[
x : \mathsf{PC},\; y : \mathsf{PC},\; \sigma : (S \times \mathsf{PC} \to \mathsf{PC}) \;\vdash\; \mathsf{transp} : \mathsf{Type}
\]
such that $\mathsf{transp}(x, y, \sigma)$ consists of all terms $s : S$ that \emph{transport} $x$ to $y$—that is, those $s$ satisfying
\[
\sigma(s, x) = y.
\]

The dependent pair type
\[
\sum_{(x, y, \sigma) : \mathsf{PC} \times \mathsf{PC} \times (S \times \mathsf{PC} \to \mathsf{PC})} \mathsf{transp}(x, y, \sigma)
\]
then consists of terms $\left((x, y, \sigma), s\right)$ such that $\sigma(s, x) = y$. In other words, a term of this type packages together the data of a source pitch class $x$, a target pitch class $y$, an $S$-action $\sigma$, and a specific element $s : S$ that transports $x$ to $y$ under $\sigma$.

(For a concrete application, see Example~\ref{ex:objectve-transformations-ti-action}. For a more elaborate use of dependent pairs, see \cite[Section 5.5.2]{flieder2025meta}, where they are used to define so-called ``self-similar networks.'')
\end{example}

\subsubsection{$\mathsf{W}$-Types}
\label{subsub:W-types}

Finally, we mention one further and particularly important class of type constructors: the $\mathsf{W}$-types. These are \emph{inductive types}, whose terms can be understood as well-founded trees determined by a specified branching structure. Intuitively, a $\mathsf{W}$-type provides a means of defining types whose elements are built inductively from a finite amount of data together with recursively defined substructures. $\mathsf{W}$-types thus underlie the definition of many familiar recursively generated data types—such as natural numbers, lists, and trees—that arise naturally in both mathematical and musical contexts (see Example~\ref{ex:rhythm-trees} for a mathemusical application to rhythm trees).

Because their construction involves additional categorical and semantic machinery, we therefore present a fuller treatment in Appendix~\ref{sec:w-types}.

\subsection{Higher-Order Formulae over a Signature}
\label{sub:higher-order-formulae}

Having introduced type constructors, we can now extend any signature $\Sigma$ by allowing these constructors to freely generate new types from the basic ones specified by $\Sigma\text{-Type}$.

A \emph{higher-order language} over~$\Sigma$ then consists of a set $\mathcal{L}(\Sigma)$ of formulae, defined recursively by the following clauses.\footnote{See \cite[Sections~D1.1 and~D4.1]{johnstone2002sketches}.}

\begin{enumerate}
  \item \emph{Relations (atomic):}  
  If $R \hookrightarrow A_1, \ldots, A_n$ is an $n$-ary relation symbol and $t_i : A_i$ for $1 \leq i \leq n$, then  
  \[
  R(t_1, \ldots, t_n) \in \mathcal{L}(\Sigma).
  \]

  \item \emph{Identity (atomic):}  
  If $s, t : A$, then  
  \[
  (s =_A t) \in \mathcal{L}(\Sigma).
  \]

  \item \emph{Membership (atomic):}  
  If $t : A$ and $P : A \to \mathsf{Prop}$, then  
  \[
  P(t) \in \mathcal{L}(\Sigma)
  \]
  (equivalently, we may write $t \in_A P$).

  \item \emph{Truth:}  
  \[
  \top \in \mathcal{L}(\Sigma).
  \]

  \item \emph{Falsehood:}  
  \[
  \bot \in \mathcal{L}(\Sigma).
  \]

  \item \emph{Conjunction:}  
  If $\phi, \psi \in \mathcal{L}(\Sigma)$, then  
  \[
  \phi \land \psi \in \mathcal{L}(\Sigma).
  \]

  \item \emph{Disjunction:}  
  If $\phi, \psi \in \mathcal{L}(\Sigma)$, then  
  \[
  \phi \lor \psi \in \mathcal{L}(\Sigma).
  \]

  \item \emph{Implication:}  
  If $\phi, \psi \in \mathcal{L}(\Sigma)$, then  
  \[
  \phi \Rightarrow \psi \in \mathcal{L}(\Sigma).
  \]

  \item \emph{Negation:}  
  If $\phi \in \mathcal{L}(\Sigma)$, then  
  \[
  \neg \phi \in \mathcal{L}(\Sigma).
  \]

  \item \emph{Universal quantification:}  
  If $\phi \in \mathcal{L}(\Sigma)$ and $x$ is a variable of type $\tau$, then  
  \[
  \forall x : \tau.\, \phi \in \mathcal{L}(\Sigma).
  \]

  \item \emph{Existential quantification:}  
  If $\phi \in \mathcal{L}(\Sigma)$ and $x$ is a variable of type $\tau$, then  
  \[
  \exists x : \tau.\, \phi \in \mathcal{L}(\Sigma).
  \]
\end{enumerate}

Variables may range over any type generated from the signature (e.g.\ basic types $A$, function types $A \to B$, dependent types $\prod_{x : A}B(x)$, etc.), which is what makes the language \emph{higher-order}. We use the usual notions of free and bound variables: quantifiers bind occurrences of the indicated (typed) variable.

\subsection{The Internal Language of a Category}
\label{sub:the-internal-language-of-a-category}

We have thus seen some initial glimpses of the power of type theory for purposes of general mathemusical reasoning. In Section~\ref{sub:a-symbolic-lingua-franca-for-mathematical-music-theorists} we observed that many of the most expressive categorical approaches to mathematical music theory employ functor categories whose internal mechanics are often highly complex. Type theory, by contrast, offers a way to perform mathematical constructions more directly, allowing us to reason and build using the simple syntactic machinery of terms, types, functions, relations, and type constructors.  

A category with suitable structure—and the musical frameworks listed in Section~\ref{sub:a-symbolic-lingua-franca-for-mathematical-music-theorists} indeed possess such structure—can itself be interpreted as a type theory by regarding the objects of the category as \emph{types}, the morphisms as \emph{function symbols}, and the subobjects $R \hookrightarrow A$ as \emph{relation symbols}. This correspondence is known as the \emph{internal language} of a category.\footnote{For the internal–external distinction, see \cite[Chapter 7]{goldblatt1984topoi} and \cite[Chapter V]{maclane2012sheaves}.}

Thus, when working with presheaves over a category $\mathcal{C}$, writing an ``element'' of a presheaf $X$ in $[\mathcal{C}^{\mathrm{op}}, \mathbf{Set}]$ as $x \in X$ constitutes an invocation of the internal language of $[\mathcal{C}^{\mathrm{op}}, \mathbf{Set}]$. This typically corresponds to forming a typing judgment $x : X$, indicating that $x$ is a term of type $X$, although there is a special case in which the membership relation may instead be conceived propositionally (see Example~\ref{ex:interpretation-of-scale-membership}).

\section{Categorical Semantics}
\label{sec:categorical-semantics}

Thus far we have introduced some of the core foundations of type theory. The reader may have noticed, however, that we have made no reference to any structure possessed by the types themselves. A signature, after all, specifies only the basic building blocks for constructing a purely formal language, whose expressions are syntactic rather than semantic.

We now turn to the question of how to endow these abstract symbolic types with meaning. This process is called providing an \emph{interpretation} or \emph{model} of the type system. Concretely, it consists in associating the abstract types with actual mathematical objects, and doing so in a consistent way: for example, ensuring that a product of types corresponds to a product of objects in the interpreting category.

For a general treatment of interpretations in categorical logic, see \cite{jacobs1999categorical,johnstone2002sketches}. For our purposes, we will restrict our attention to interpretations in toposes, as they provide interpretations of higher-order logic.  To achieve this, we introduce the concept of a $\Sigma $-\emph{structure}.\footnote{See \cite[Section D1.2]{johnstone2002sketches} for the following definition.}

\begin{definition}[$\Sigma$-Structure]
\label{def:sigma-structure}
Let $\Sigma$ be a signature and $\mathcal{E}$ a topos.  
A $\Sigma$-\emph{structure} in $\mathcal{E}$ is an assignment $M$ that interprets each type, function symbol, and relation symbol of $\Sigma$ as an object or morphism of $\mathcal{E}$ in a way that is compatible with the type constructors, as follows:
\begin{enumerate}
    \item For each type $A \in \Sigma\text{-Type}$, $M$ assigns an object $M(A) \in \mathcal{E}$.  
This assignment must respect the structure of type constructors; for example,
\[
M(A \to B) = M(B)^{M(A)}, \qquad 
M(A \times B) = M(A) \times M(B), \qquad 
\]
\[
M(\mathbf{1}) = 1, \qquad 
M(\mathsf{Prop}) = \Omega, \qquad \text{etc.}
\]

    \item For each function symbol $f : A \to B$ in $\Sigma\text{-Fun}$, $M$ assigns a morphism
    \[
    M(f) : M(A) \longrightarrow M(B).
    \]
    
    \item For each relation symbol $R \hookrightarrow A_1, \ldots, A_n$ in $\Sigma\text{-Rel}$, $M$ assigns a subobject
    \[
    M(R) \hooklongrightarrow M(A_1, \ldots, A_n),
    \]
    which, by the universal property of the subobject classifier, corresponds to a morphism
    \[
    M(R) : M(A_1, \ldots, A_n) \longrightarrow \Omega.
    \]
\end{enumerate}

The $\Sigma$-structures in $\mathcal{E}$ form a category
$\Sigma\text{-}\mathbf{Str}(\mathcal{E})$ whose morphisms (``$\Sigma$-structure
homomorphisms'') $h : M \to N$ consist of a family of maps $h_A : M(A) \to N(A)$ for each $A \in \Sigma\text{-Type}$ which are compatible with the type constructors, and such that:
\begin{enumerate}
\item (Functions commute) For every function symbol $f: A \to B$ in
$\Sigma\text{-Fun}$, the square
\[\begin{tikzcd}
	{M(A)} && {M(B)} \\
	\\
	{N(A)} && {N(B)}
	\arrow["{M(f)}", from=1-1, to=1-3]
	\arrow["{h_A}"', from=1-1, to=3-1]
	\arrow["{h_B}", from=1-3, to=3-3]
	\arrow["{N(f)}"', from=3-1, to=3-3]
\end{tikzcd}\]
commutes.

\item (Relations are preserved) For every relation symbol
$R \hookrightarrow A_1,\dots,A_n$ in $\Sigma\text{-Rel}$, there exists a unique arrow $h_R : M(R) \to N(R)$ making
\[\begin{tikzcd}
	{M(R)} && {M(A_1, \ldots, A_n)} \\
	\\
	{N(R)} && {N(A_1, \ldots, A_n)}
	\arrow[hook, from=1-1, to=1-3]
	\arrow["{h_R}"', from=1-1, to=3-1]
	\arrow["{h_{A_1} \; \times \; \cdots \; \times h_{A_n}}", from=1-3, to=3-3]
	\arrow[hook, from=3-1, to=3-3]
\end{tikzcd}\]
commute.
\end{enumerate}
\end{definition}

Now that we have formalized the notion of a $\Sigma$-structure, we can relate it to the classifying category. Given a signature $\Sigma$, recall (Section \ref{sub:the-classifying-category-of-a-signature}) that we constructed its classifying category $\mathcal{C}\ell(\Sigma)$. A $\Sigma$-structure is precisely a functor
\[
M : \mathcal{C}\ell(\Sigma) \longrightarrow \mathcal{E}.
\]
This functor interprets the objects (contexts) and morphisms (typing judgments) of $\mathcal{C}\ell(\Sigma)$ in the topos $\mathcal{E}$, thereby providing a semantics for the syntactic expressions determined by the signature $\Sigma$.

We denote the interpretation of a type, function symbol, or relation symbol $X$ by $M$ as either $M(X)$, $\llbracket X \rrbracket_M$, or simply $\llbracket X \rrbracket$ whenever $M$ is clear from context.

For reasons of space, we defer the categorical semantics of the logical connectives and quantifiers to Appendix~\ref{sec:categorical-semantics-a-deeper-look}. For the present, we conclude this section with examples of how some of the purely type-theoretic constructions introduced in Sections~\ref{subsub:typing-judgments-for-data-types} and~\ref{subsub:typing-judgments-for-propositions} may be interpreted semantically.

\begin{example}[Interpretation of pitch classes and intervals]
\label{ex:interpretation-of-pitch-classes-and-intervals}
In Example~\ref{ex:constructing-intervals} we introduced a type $\mathsf{Pitch}$ of pitches and a type $\mathsf{IVLS}$ of intervals, showing how to construct an interval $i : \mathsf{IVLS}$ from pitches $x, y : \mathsf{Pitch}$. Suppose now that these types belong to a signature $\Sigma$, and consider an interpretation into the category of sets:
\[
M : \mathcal{C}\ell(\Sigma) \longrightarrow \mathbf{Set}.
\]
We may set $\llbracket \mathsf{Pitch} \rrbracket = \Z_{12}$, the standard set-theoretic representation of the 12 pitch classes, and likewise $\llbracket \mathsf{IVLS} \rrbracket = \Z_{12}$, the set of 12 pitch-class intervals. This yields the familiar set-theoretic interpretation of the interval function:
\[
\begin{matrix}
\llbracket i \rrbracket = \mathrm{int} : & \Z_{12} \times \Z_{12} & \longrightarrow & \Z_{12} \\
& (x, y) & \longmapsto & y - x.
\end{matrix}
\]

To enrich these sets with additional structure, we may instead interpret the pitch and interval types in $\RelAt$. For example, by interpreting $\mathsf{IVLS}$ as a group and imposing the relevant axioms, we recover Lewin's Definition~2.3.1 from \cite{lewin2007generalized} of a \emph{generalized interval system} (GIS). 

Another way to define a GIS is to specify it directly as a theory~$\mathbb{T}$ over a given signature~$\Sigma$—that is, as a set of formulae in~$\mathcal{L}(\Sigma)$ serving as the axioms of the theory (see~\cite[Section D1.1]{johnstone2002sketches}). Once the theory is fixed, its models automatically satisfy the structural requirements of a GIS. This perspective illustrates the power of the type-theoretic approach: the axioms of a GIS are not imposed externally but are expressed intrinsically within the theory itself, and models of the theory are functors that incarnate the structural requirements specified by the axioms in a concrete setting: namely, as objects and morphisms in the codomain of the functor.

Let us therefore define the theory~$\mathbb{GIS}$ of a generalized interval system.  
Following~\cite[26]{lewin2007generalized}, a GIS is defined as a triple  
\[
(S, \mathsf{IVLS}, \mathsf{int})
\]
where $S$ is the \emph{space} of the system (a set of elements), $\mathsf{IVLS}$ is the \emph{group of intervals} (a group), and $\mathsf{int} : S \times S \to \mathsf{IVLS}$ is a function satisfying two axioms:
\begin{enumerate}
\item For every $r, s, t : S$, we have $\mathsf{int}(r, s) \star \mathsf{int}(s, t) = \mathsf{int}(r, t)$. 
\item For every $s : S$ and every $i : \mathsf{IVLS}$, there is a unique $t : S$ such that $\mathsf{int}(s, t) = i$. 
\end{enumerate}

This can be formalized as a theory as follows. The basic types of its signature $\Sigma$  are $S$ and $\mathsf{IVLS}$. The type~$S$ carries no additional structure, whereas~$\mathsf{IVLS}$ is endowed with the structure of a group, which can be expressed in terms of the equational identities discussed in Example~\ref{ex:signature-of-group}.

The next essential step is to encode conditions (1) and (2). Both are expressible as commutativity conditions on specific diagrams. Condition (1)—the composition property—requires that the following diagram commute:
\[\begin{tikzcd}
	{S \times S \times S} &&&& {\mathsf{IVLS} \times \mathsf{IVLS}} \\
	\\
	\\
	{S \times S} &&&& {\mathsf{IVLS}}
	\arrow["{\langle \mathsf{int} \ \circ \ \mathsf{pr}_{1,2}, \; \mathsf{int} \ \circ \ \mathsf{pr}_{2, 3} \rangle}", from=1-1, to=1-5]
	\arrow["{\mathsf{pr}_{1,3}}"', from=1-1, to=4-1]
	\arrow["\star", from=1-5, to=4-5]
	\arrow["{\mathsf{int}}"', from=4-1, to=4-5]
\end{tikzcd}\]
Here $\mathsf{pr}_{1,2}$, $\mathsf{pr}_{2,3}$, and $\mathsf{pr}_{1,3}$ denote the evident projections from $S \times S \times S$. This expresses precisely the requirement that for all $r, s, t : S$,
\[
\mathsf{int}(r, s) \star \mathsf{int}(s, t) = \mathsf{int}(r, t).
\]

Condition (2)—existence and uniqueness—is expressed by specifying a function 
\[
\begin{matrix}
u : & S \times \mathsf{IVLS} & \longrightarrow & S \\
& (s, i) & \longmapsto & t
\end{matrix}
\]
such that the following diagram commutes:
\[\begin{tikzcd}
	{S \times \mathsf{IVLS}} && {S \times S} \\
	\\
	&& {\mathsf{IVLS}}
	\arrow["{\langle \mathsf{pr}_1, \; u \rangle}", from=1-1, to=1-3]
	\arrow["{\mathsf{pr}_2}"', from=1-1, to=3-3]
	\arrow["{\mathrm{int}}", from=1-3, to=3-3]
\end{tikzcd}\]
This expresses the requirement that for each $s : S$ and $i : \mathsf{IVLS}$ there is a unique $t$ such that $\mathsf{int}(s,t) = i$.

Thus, defining the theory $\mathbb{GIS}$ this way ensures that every model of the theory \emph{is} a generalized interval system. Interpreting the signature in $\mathbf{Set}$ recovers the familiar set-theoretic conception, while interpreting it in a category whose objects carry additional structure makes it possible to define new kinds of generalized interval systems. 
\end{example}

\begin{example}[Interpretation of scale membership]
\label{ex:interpretation-of-scale-membership}
In Example~\ref{ex:scale-membership} we introduced a type $\mathsf{PC}$ of pitch classes and a type $\mathsf{Scale}$ of scales, together with a membership relation $(p \in s)$ in the context $p : \mathsf{PC}, \; s : \mathsf{Scale}$, i.e.
\[
p : \mathsf{PC}, \; s : \mathsf{Scale} \;\vdash\; (p \in s) : \mathsf{Prop}.
\]

For the membership relation to be meaningful, the terms of $\mathsf{Scale}$ must in some way be built from terms of $\mathsf{PC}$. A natural approach is to regard scales as sets of pitch classes, corresponding to propositional functions $s : \mathsf{PC} \to \mathsf{Prop}$, where $s(x)$ holds precisely when $x$ belongs to the scale. Thus we define
\[
\mathsf{Scale} \coloneqq \mathsf{PC} \longrightarrow \mathsf{Prop}.
\]

With this definition, the membership typing judgment
\[
x : \mathsf{PC}, \; s : \mathsf{PC} \to \mathsf{Prop} \;\vdash\; (x \in s) : \mathsf{Prop}
\]
corresponds to the evaluation map
\[
\begin{matrix}
\mathsf{eval} : & (\mathsf{PC} \to \mathsf{Prop}) \times \mathsf{PC} & \longrightarrow & \mathsf{Prop} \\
& (s, x) & \longmapsto & s(x).
\end{matrix}
\] 

Notice therefore that type theory affords two different notions of ``membership.''
\begin{itemize}  
\item \emph{Quantifier-bounding membership} is expressed by typing judgments like $x : X$, which correspond to morphisms $() \to X$ in the classifying category. This is not a proposition, but a declaration that $x$ is given as a term of $X$.  
\item \emph{Propositional membership} is expressed by a typing judgment of the form
\[
x : X, \; A : X \to \mathsf{Prop} \;\vdash\; x \in A : \mathsf{Prop},
\]
which is capable of being true or false.\footnote{The terms ``quantifer-bounding membership'' and ``propositional membership'' are from \cite{shulman2013set}.}
\end{itemize}

In a presheaf category $[\mathcal{C}^{\mathrm{op}}, \mathbf{Set}]$, the notation $ x  \in  A $ where $ A $ is a presheaf can therefore be understood in two ways:
\begin{enumerate}
\item as a typing judgment $() \; \vdash \; x : A$, i.e.\ a morphism $ x  : \mathbf{1} \to  A $; or
\item if $A$ is a propositional function $A : X \to \Omega$, then $x \in A$ denotes a membership relation between presheaves $ X $ and $\Omega^{ X }$, i.e.\ a morphism 
\[
 \in \;  \colon  X  \times \Omega^{ X } \longrightarrow \Omega.
\]
\end{enumerate}
\end{example}

This suffices for an initial overview of categorical semantics. For a fuller treatment, see Appendix~\ref{sec:categorical-semantics-a-deeper-look}. 

\section{An Application to the Theory of Voice-Leading Spaces}
\label{sec:an-application-to-the-theory-of-voice-leading-spaces}

In this section, we present a nontrivial application of the type-theoretic machinery to the general formulation of voice-leading spaces, with the aim of demonstrating its expressive power. The approach adopted here is agnostic with respect to the nature of the transformations between voices and therefore affords a high degree of generality. Tymoczko \cite{tymoczko2025concept} has argued  that transformational theory, insofar as it models transformations between musical objects as group actions on sets, is inherently limited. Such an approach cannot capture certain transformations—most notably homotopy classes of paths in a topological space—which Tymoczko regards as central to musical thought. As an alternative, he suggests modeling voice-leading spaces by groupoids, in which the objects are musical entities such as chords and the morphisms are homotopy-equivalence classes of paths between them.

At the same time, there remain situations in which more traditional transformational networks are of interest, where the objects are set-like and the relations between them function-like.

The framework we develop here accommodates both approaches, and more. By adopting the type-theoretic formulation we sidestep unnecessary ontological disputes, as the framework does not commit us to any restricted class of objects, but instead specifies the syntactic form required for a construction to qualify as a voice-leading space. The formulation therefore provides a high-level specification of the concept of a voice-leading space, delineating the minimal conditions under which an object may be recognized as such.

\subsection{The Type of Voice-Leading Spaces}

In this section, we work toward formulating the type of voice-leading spaces. We proceed in phases. Suppose first that we have a type $\mathsf{Pitch}$ of pitched objects. For instance, these could be frequencies, pitch classes, $n$-tuples of pitch classes, and so forth. These will serve as the objects between which voice leadings occur. 

The first step in defining a voice-leading space is to specify a \emph{voice-leading rule}, namely a function
\[
\mathsf{vlr} : \mathsf{Pitch} \times \mathsf{Pitch} \longrightarrow \mathsf{Type}
\]
that assigns to each pair $(x, y) : \mathsf{Pitch} \times \mathsf{Pitch}$ the type of \emph{ways of getting from $x$ to $y$}. 

We may think of each way of getting from $x$ to $y$ as an arrow $x \to y$, so that a voice-leading space has the structure of a quiver (a directed graph allowing multiple arrows between any two vertices). A \emph{quiver} is defined formally as a quadruple $(A, V, s, t)$ where $A$ is a type of arrows, $V$ a type of vertices, and 
\[
s, t : A \longrightarrow V
\]
are the source and target maps. 

A voice-leading space will therefore be a kind of quiver, and hence a term of a quiver type. To define such a type, we must briefly examine some details concerning type universes. Thus far we have spoken informally, writing expressions such as $A : \mathsf{Type}$ to indicate that $A$ is a type. Of course, if $\mathsf{Type} : \mathsf{Type}$ without restriction, we would encounter Russell's paradox. Instead, what is meant is that $\mathsf{Type}$ refers to a universe of types up to a certain level in a hierarchy of universes. Formally, we posit a cumulative hierarchy of universes
\[
\mathsf{Type}_0 : \mathsf{Type}_1 : \mathsf{Type}_2 : \cdots
\]
where each universe $\mathsf{Type}_i$ is itself a term of $\mathsf{Type}_{i+1}$, and whenever $A : \mathsf{Type}_i$ it also follows that $A : \mathsf{Type}_{i+1}$.\footnote{See \cite[Section 1.3]{hott2013}.} By the convention of \emph{typical ambiguity} (see \cite[p.\ 24]{hott2013}), however, we may simply write $\mathsf{Type} : \mathsf{Type}$, to be understood as shorthand for $\mathsf{Type}_i : \mathsf{Type}_{i+1}$. 

With this in place, we can now define the type of quivers.  A quiver is a quadruple $(A, V, s, t)$, where $A$ and $V$ are arbitrary types and $s, t : A \to V$.  Accordingly, we define the type of quivers as the doubly dependent sum:
\[
\mathsf{Quiv} \;\coloneqq\; \sum_{A : \mathsf{Type}} \sum_{V : \mathsf{Type}} (A \to V) \times (A \to V).
\]

We now show how to construct voice-leading spaces as terms of $\mathsf{Quiv}$. 

\begin{definition}[Voice-Leading Space]
Let $\mathsf{Pitch}$ be a pitch type and 
\[
\mathsf{vlr} : \mathsf{Pitch} \times \mathsf{Pitch} \longrightarrow \mathsf{Type}
\]
a voice-leading rule, i.e.\ a function assigning to each pair $(x, y)$ of pitch terms the type of \emph{ways of getting from $x$ to $y$}. The \emph{voice-leading space} over $(\mathsf{Pitch}, \mathsf{vlr})$ is the quiver 
\[
\mathsf{vls}(\mathsf{Pitch}, \mathsf{vlr}) = \left( \sum_{(x, y) : \mathsf{Pitch} \times \mathsf{Pitch}} \mathsf{vlr}(x, y), \; \mathsf{Pitch}, \; \mathsf{pr}_1 \circ \mathsf{pr}_1, \; \mathsf{pr}_2 \circ \mathsf{pr}_1 \right).
\]
That is, the source map 
\[
\begin{matrix}
\mathsf{pr}_1 \circ \mathsf{pr}_1 : & \sum\limits_{(x, y) : \mathsf{Pitch} \times \mathsf{Pitch}} \mathsf{vlr}(x, y) & \longrightarrow & \mathsf{Pitch} \\
& \left( (x, y), t \right) & \longmapsto & x
\end{matrix}
\]
takes $t$ (a way of getting from $x$ to $y$) to $x$, while the target map $\mathsf{pr}_2 \circ \mathsf{pr}_1$ sends $t$ to $y$. 

Thus, 
\[
\mathsf{vls}(\mathsf{Pitch}, \mathsf{vlr}) : \mathsf{Quiv}.
\]
\end{definition}

By definition, constructing a voice-leading space depends on the data $\mathsf{Pitch}$ and $\mathsf{vlr}$, which are terms. We therefore seek a context for the typing judgment
\[
\mathsf{vls}(\mathsf{Pitch}, \mathsf{vlr}) : \mathsf{Quiv}.
\]

Let $\mathsf{PitchType}$ be a type of pitch types, so that $\mathsf{Pitch} : \mathsf{PitchType}$. To each $P : \mathsf{PitchType}$ we associate the type of voice-leading rules on $P$, given by
\[
\begin{matrix}
\mathsf{vlrules} : & \mathsf{PitchType} & \longrightarrow & \mathsf{Type} \\
& P & \longmapsto & \left((P \times P) \to \mathsf{Type}\right),
\end{matrix}
\]
where a term of type $\left((P \times P) \to \mathsf{Type}\right)$ is a function assigning to each pair $(x, y) : P \times P$ the type of all ways of getting from $x$ to $y$.

Hence, the dependent pair type
\[
\sum_{P : \mathsf{PitchType}} \mathsf{vlrules}(P)
\]
is the type of pairs $(P, V_P)$ where $P$ is a pitch type and $V_P$ a voice-leading rule for $P$. 

As we saw above, the voice-leading space for such a pair $(P, V_P)$ is $\mathsf{vlspace}(P, V_P)$. This yields the general procedure for constructing a voice-leading space:
\[
(p, v_p) : \sum_{P : \mathsf{PitchType}} \mathsf{vlrules}(P) \;\vdash\; \mathsf{vls}(p, v_p) : \mathsf{Quiv}.
\]
That is, given a pitch type $p$ and a voice-leading rule $v_p$ on $p$, the construction $\mathsf{vls}(p, v_p)$ yields the quiver encoding its voice-leading space. 

Equivalently, in the classifying category this corresponds to the morphism
\[
\mathsf{vls} : \sum_{P : \mathsf{PitchType}} \mathsf{vlrules}(P) \longrightarrow \mathsf{Quiv}.
\]

\subsection{Semantics for Voice-Leading Spaces}

Observe that the construction of a voice-leading space has thus far been purely syntactic, remaining agnostic with respect to the specific structures carried by the pitch types and voice-leading rules. We now present semantically interpreted examples, in which the syntactic constructions are modeled within a category possessing sufficient structure—for instance, the presheaf topos $\RelAt$~\cite{flieder2024towards}, which supports the kinds of structural information required for our purposes. For clarity, we stipulate the relevant structures \emph{by fiat}, rather than invoking the full presheaf construction, so as to keep the focus on the underlying conceptual ideas.

\begin{example}[Voice-leadings as group actions]
\label{ex:voice-leadings-as-group-actions}
Consider a voice-leading space 
\[
\mathsf{vls}(\mathsf{Pitch}, \mathsf{vlr}),
\]
which we first define syntactically. Suppose a type $G$ acts on $\mathsf{Pitch}$ by
\[
\varphi : G \times \mathsf{Pitch} \longrightarrow \mathsf{Pitch}, \qquad \varphi(g, x) = gx.
\]

We then define the voice-leading rule
\[
\mathsf{vlr} : \mathsf{Pitch} \times \mathsf{Pitch} \longrightarrow \mathsf{Type}
\]
that sends each pair $(x, y) : \mathsf{Pitch} \times \mathsf{Pitch}$ to the subtype of $G$ consisting of those $g$ with $gx = y$. We may call this the type of \emph{transporters} from $x$ to $y$, as discussed in Example \ref{ex:transporters}. 

For example, let $\llbracket \mathsf{Pitch} \rrbracket = \Z_{12}$ be the set of twelve pitch classes and let $\llbracket G \rrbracket = T/I$ be the group of transpositions and inversions acting on $\Z_{12}$. Under this interpretation, the resulting voice-leading space is the quiver whose vertices are the pitch classes in $\Z_{12}$ and whose arrows are the transposition and inversion operators.
\end{example}

\begin{example}[Voice-leadings as homotopy equivalence classes]
As a contrasting case, suppose we interpret 
\[
\llbracket \mathsf{Pitch} \rrbracket = S^1
\]
as the unit circle, representing continuous pitch classes topologically. Let $\Pi(S^1)$ denote the fundamental groupoid of $S^1$, whose objects are the points of $S^1$ and whose morphisms are homotopy-equivalence classes of paths between them. We write $\mathrm{Hom}_{\Pi(S^1)}(x, y)$ for the set of homotopy-equivalence classes of paths from $x$ to $y$ in~$S^1$. 

The voice-leading rule is then interpreted as the morphism
\[
\begin{matrix}
\llbracket \mathsf{vlr} \rrbracket : & S^1 \times S^1 & \longrightarrow & \mathrm{Hom}_{\Pi(S^1)}(-, -) \\
& (x, y) & \longmapsto & \mathrm{Hom}_{\Pi(S^1)}(x, y)
\end{matrix}
\]
assigning to each pair $(x, y) \in S^1$ the set of homotopy-equivalence classes of paths from $x$ to $y$.

The resulting voice-leading space is a quiver whose vertices are the points $x \in S^1$ and whose arrows are homotopy-equivalence classes of paths between them.

More generally, if $\mathsf{Pitch}$ is interpreted as any topological space whose points represent pitched objects, the corresponding voice-leading space is obtained in the same way, with arrows given by homotopy-equivalence classes of paths. Thus, the topological voice-leading spaces central to Tymoczko’s work\footnote{See, for example, \cite{tymoczko2010geometry,tymoczko2023tonality}.} are naturally encompassed within our framework.
\end{example}

\subsection{Subjective and Objective Transformations}

In this section we seek to explicate the notions of \emph{subjective} and \emph{objective} transformations in the context of voice-leading spaces, following \cite{tymoczko2025concept}. Subjective transformations correspond to transformations \emph{within} a voice-leading space, whereas objective transformations correspond to transformations \emph{of} the space itself. An objective transformation is therefore a self-map of a voice-leading space that preserves its structure, so as to maintain the information encoded therein. To make this precise, we first recall the definition of a quiver homomorphism.  

A \emph{quiver homomorphism}
\[
\Gamma \coloneq (\gamma_1, \gamma_0) : (A, V, s, t) \longrightarrow (A', V', s', t')
\] 
consists of functions on arrows and vertices such that the diagram
\[\begin{tikzcd}
	A && {A'} \\
	\\
	V && {V'}
	\arrow["{\gamma_1}", from=1-1, to=1-3]
	\arrow["s"', shift right=2, from=1-1, to=3-1]
	\arrow["t", shift left=2, from=1-1, to=3-1]
	\arrow["{s'}"', shift right=2, from=1-3, to=3-3]
	\arrow["{t'}", shift left=2, from=1-3, to=3-3]
	\arrow["{\gamma_0}"', from=3-1, to=3-3]
\end{tikzcd}\]
commutes; that is, $\gamma_1$ respects source and target with respect to $\gamma_0$.  

A quiver isomorphism is a quiver homomorphism in which $\gamma_1$ and $\gamma_0$ are isomorphisms. An automorphism is an isomorphism from a quiver to itself.  

With these preliminaries in place, we can now define subjective and objective transformations.
 
\begin{definition}[Subjective and Objective Transformations]
\label{def:subjective-and-objective-transformations}
Let $\mathsf{vls}(P, V_P) : \mathsf{Quiv}$ be the voice-leading space
\[
\mathsf{vls}(P, V_P) = \left( \sum_{(x, y) : P \times P} V_P(x, y), \; P, \; \mathsf{pr}_1 \circ \mathsf{pr}_1, \; \mathsf{pr}_2 \circ \mathsf{pr}_1 \right). 
\]
We distinguish two kinds of transformations:
\begin{enumerate}
\item A \emph{subjective transformation in} $\mathsf{vls}(P, V_P)$ is an arrow of the quiver, i.e.
\[
\alpha : \sum_{(x, y) : P \times P} V_P(x, y).
\] 
\item An \emph{objective transformation on} $\mathsf{vls}(P, V_P)$ is a quiver automorphism
\[
\Gamma : \mathsf{vls}(P, V_P) \xlongrightarrow{\cong} \mathsf{vls}(P, V_P).
\]
\end{enumerate}
\end{definition}

\begin{example}[Objective transformations of the $T/I$-action groupoid]
\label{ex:objectve-transformations-ti-action}
In Example~\ref{ex:voice-leadings-as-group-actions} we described the $T/I$ group’s action on $\Z_{12}$ as a voice-leading space whose vertices are the elements of $\Z_{12}$ and whose arrows are the elements of $T/I$. This is precisely the \emph{action groupoid} $\Z_{12} \sslash (T/I)$, a category (hence a quiver) modeling the action of a group on a set. 

Every group element $\phi \in T/I$ induces a quiver automorphism $\Phi = (\phi_1, \phi_0)$ defined on vertices by $\phi_0(x) = \phi(x)$ and on arrows $x \xrightarrow{g} y$ by conjugation,
\[
\phi_1(g) = \phi g \phi^{-1}.
\]

Let $\mathrm{Transp}_{T/I}(x, y)$ denote the transporter set of elements $g \in T/I$ with $g(x) = y$. Then the arrow object of this voice-leading space is
\[
V_{\Z_{12}} \;=\; \sum_{(x, y) \in \Z_{12} \times \Z_{12}} \mathrm{Transp}_{T/I}(x, y).
\]

Thus an objective transformation induced by $\phi$ is represented by a quiver automorphism. In other words, it maps subjects (vertices) to subjects, and subjective transformations (arrows) to subjective transformations, preserving the structure of the voice-leading space:
\[\begin{tikzcd}
	{\left( (x, y), \; g \right)} && {\left( (\phi x, \phi y), \; \phi g \phi^{-1} \right)} \\
	{V_{\Z_{12}}} && {V_{\Z_{12}}} \\
	\\
	\\
	{\Z_{12}} && {\Z_{12}} \\
	a && {\phi a}
	\arrow[maps to, from=1-1, to=1-3]
	\arrow["{\phi_1}", from=2-1, to=2-3]
	\arrow["s"', shift right=2, from=2-1, to=5-1]
	\arrow["t", shift left=2, from=2-1, to=5-1]
	\arrow["s"', shift right=2, from=2-3, to=5-3]
	\arrow["t", shift left=2, from=2-3, to=5-3]
	\arrow["{\phi_0}"', from=5-1, to=5-3]
	\arrow[maps to, from=6-1, to=6-3]
\end{tikzcd}\]
\end{example}

\subsection{An Alternative Approach: The Signature of a Voice-Leading Space}

The constructions above define voice-leading spaces as terms of a type $\mathsf{Quiv}$, given by a pair $(\mathsf{Pitch}, \mathsf{vlr})$ consisting of a pitch type and a voice-leading rule, together with the operation $\mathsf{vls}$, which constructs the corresponding quiver. Alternatively, the same concept may be described by introducing a signature $\Sigma_{\mathsf{VLS}}$ whose symbols specify the data of a voice-leading space. A $\Sigma_{\mathsf{VLS}}$-structure in a topos $\mathcal{E}$ is then an interpretation of this signature in $\mathcal{E}$, and thus corresponds, via the constructor $\mathsf{vls}$, to an actual voice-leading space internal to $\mathcal{E}$. Morphisms between such structures correspond to voice-leading space homomorphisms, while automorphisms of a $\Sigma_{\mathsf{VLS}}$-structure coincide with the \emph{objective transformations} defined above.

\paragraph{The signature.}
The signature $\Sigma_{\mathsf{VLS}}$ consists of two basic types:
\begin{enumerate}
\item A type $\mathsf{Pitch}$ of pitch data.
\item A type $\mathsf{Arrow}$ of arrow data.
\end{enumerate}
From these, we form the compound type $\mathsf{Pitch} \times \mathsf{Pitch}$ and the power object $P\mathsf{Arrow}$, whose elements correspond to subtypes of $\mathsf{Arrow}$—that is, to subobjects of arrows.\footnote{Throughout the text, such subobjects are denoted by the function type $\mathsf{Arrow} \to \mathsf{Prop}$, reflecting the equivalence between propositional functions and subobjects.}

The voice-leading rule is represented by a function symbol
\[
\mathsf{vlr} : \mathsf{Pitch} \times \mathsf{Pitch} \longrightarrow P\mathsf{Arrow},
\]
assigning to each ordered pair $(x,y)$ of pitches the subtype of $\mathsf{Arrow}$ consisting of the ways of moving from $x$ to $y$.

\paragraph{Definable quiver.}
From this minimal signature, a quiver object can be defined internally to $\mathcal{E}$.  For any $\Sigma_{\mathsf{VLS}}$-structure $M : \mathcal{C}\ell(\Sigma_{\mathsf{VLS}}) \to \mathcal{E}$, the corresponding voice-leading space $\mathsf{vls}(M)$ is the quiver
\[\begin{tikzcd}
	{\sum\limits_{(x, y) \; \in \; \llbracket \mathsf{Pitch} \rrbracket_M \; \times \; \llbracket \mathsf{Pitch} \rrbracket_M }  \llbracket \mathsf{vlr} \rrbracket_M(x,y)} &&& {\llbracket \mathsf{Pitch} \rrbracket_M}
	\arrow["{\mathrm{pr}_1 \; \circ \; \mathrm{pr}_1}", shift left=2, from=1-1, to=1-4]
	\arrow["{\mathrm{pr}_2 \; \circ \; \mathrm{pr}_1}"', shift right=2, from=1-1, to=1-4]
\end{tikzcd}\]
internal to $\mathcal{E}$.

\begin{proposition}
For a category $\Sigma_\mathsf{VLS}\text{-}\mathbf{Str}(\mathcal{E})$ of $\Sigma_\mathsf{VLS}$-structures, an \emph{objective transformation} of a voice-leading space is induced by an automorphism $h : M \to N$ in $\Sigma_\mathsf{VLS}\text{-}\mathbf{Str}(\mathcal{E})$. 
\end{proposition}

\begin{proof}
Since $\mathsf{vlr}$ is the only nontrivial function symbol in $\Sigma_{\mathsf{VLS}}$, a $\Sigma_{\mathsf{VLS}}$-isomorphism $h : M \to N$ consists of isomorphisms $h_\mathsf{Pitch} : \llbracket \mathsf{Pitch} \rrbracket_M \to \llbracket \mathsf{Pitch} \rrbracket_N$ and $h_\mathsf{Arrow} : \llbracket \mathsf{Arrow} \rrbracket_M \to \llbracket \mathsf{Arrow} \rrbracket_N$ such that the square
\[\begin{tikzcd}
	{\llbracket \mathsf{Pitch} \rrbracket_M \times \llbracket \mathsf{Pitch} \rrbracket_M} && {\llbracket P\mathsf{Arrow} \rrbracket_M} \\
	\\
	{\llbracket \mathsf{Pitch} \rrbracket_N \times \llbracket \mathsf{Pitch} \rrbracket_N} && {\llbracket P\mathsf{Arrow} \rrbracket_N}
	\arrow["{\llbracket \mathsf{vlr} \rrbracket_M}", from=1-1, to=1-3]
	\arrow["{h_\mathsf{Pitch} \; \times \; h_\mathsf{Pitch} }"', from=1-1, to=3-1]
	\arrow["{h_{P\mathsf{Arrow}}}", from=1-3, to=3-3]
	\arrow["{\llbracket \mathsf{vlr} \rrbracket_N}"', from=3-1, to=3-3]
\end{tikzcd}\]
commutes. This induces the quiver isomorphism
\[\begin{tikzcd}
	{\sum\limits_{(x,  y) \; \in \; \llbracket \mathsf{Pitch} \; \times \; \mathsf{Pitch} \rrbracket_M }  \llbracket \mathsf{vlr} \rrbracket_M(x,y)} &&& {\llbracket \mathsf{Pitch} \rrbracket_M} \\
	\\
	{\sum\limits_{(x', y') \; \in \; \llbracket \mathsf{Pitch} \; \times \; \mathsf{Pitch} \rrbracket_N }  \llbracket \mathsf{vlr} \rrbracket_N(x',y')} &&& {\llbracket \mathsf{Pitch} \rrbracket_N}
	\arrow["{\mathrm{pr}_1 \; \circ \; \mathrm{pr}_1}", shift left=2, from=1-1, to=1-4]
	\arrow["{\mathrm{pr}_2 \; \circ \; \mathrm{pr}_1}"', shift right=2, from=1-1, to=1-4]
	\arrow["{(h_\mathsf{Pitch} \; \times \; h_\mathsf{Pitch}) \; \times \; h_\mathsf{Arrow}}"', from=1-1, to=3-1]
	\arrow["{h_{\mathsf{Pitch}}}", from=1-4, to=3-4]
	\arrow["{\mathrm{pr}_1 \; \circ \; \mathrm{pr}_1}", shift left=2, from=3-1, to=3-4]
	\arrow["{\mathrm{pr}_2 \; \circ \; \mathrm{pr}_1}"', shift right=2, from=3-1, to=3-4]
\end{tikzcd}\]
When $M=N$, this reduces to an automorphism, i.e.\ an objective transformation of a voice-leading space.
\end{proof}

\begin{corollary}
The assignment $M \mapsto \mathsf{vls}(M)$ extends to a functor
\[
\mathsf{vls} :
  \Sigma_{\mathsf{VLS}}\text{-}\mathbf{Str}(\mathcal{E})
  \longrightarrow
  \mathbf{Quiv}(\mathcal{E}).
\]
Moreover, if some $h : M \to N$ is an automorphism in $\Sigma_\mathsf{VLS}\text{-}\mathbf{Str}(\mathcal{E})$, then 
\[
\mathsf{vls}(h) : \mathsf{vls}(M) \longrightarrow \mathsf{vls}(N)
\] 
is an objective transformation in $\mathbf{Quiv}(\mathcal{E})$.
\end{corollary}

\section{Conclusion}
\label{sec:conclusion}

In this paper we have presented type theory as a symbolic \emph{lingua franca} for mathematical music theorists. By this we mean that type theory supplies a syntactic machinery for expressing the form of mathemusical constructions even prior to their interpretation---a role fulfilled, as shown in Section~\ref{sec:categorical-semantics} (and examined more thoroughly in Appendix~\ref{sec:categorical-semantics-a-deeper-look}), by categorical semantics.

We have illustrated the framework’s generality through a range of examples: from simple musical propositions such as ``$c$ (a chord) is the dominant of $k$ (a key)'' (Example~\ref{ex:functional-harmony}), to Lewin’s conception of a generalized interval system as a theory over a signature $\mathbb{GIS}$ whose semantic models are particular generalized interval systems (Example~\ref{ex:interpretation-of-pitch-classes-and-intervals}). A more developed application was presented in Section~\ref{sec:an-application-to-the-theory-of-voice-leading-spaces}, where we formalized the concept of a voice-leading space together with the notions of subjective and objective transformations (Definition~\ref{def:subjective-and-objective-transformations}). The reader may also note Examples~\ref{ex:proof-all-interval} and~\ref{ex:proof-domfunc-leading-tone} in Section~\ref{subsub:quantifiers-in-dependent-type-theory}, which show how mathemusical propositions---such as ``$C$ is an all-interval pitch-class set''---correspond to types whose terms are proofs of those propositions. This perspective, known as \emph{propositions-as-types}, is not only philosophically illuminating but also practically fruitful, since it permits the construction of types defined by the procedures required to verify the truth of a proposition: to construct a proof is precisely to exhibit a term of such a type.

The long-term value of the type-theoretic framework will ultimately depend on the fruitfulness of future applications. While the treatment of voice-leading spaces in Section~\ref{sec:an-application-to-the-theory-of-voice-leading-spaces} already offers positive evidence of its usefulness for mathemusical reasoning, further work will be needed to determine its broader scope. My hope is that theorists will feel encouraged to explore type theory as a symbolic language serviceable to the formulation of mathemusical concepts. 

Many categorical approaches are already in use (e.g.~\cite{flieder2024towards,noll2005topos,mazzola2018topos,mazzola2002topos,popoff2018relational}); and, as discussed in Section~\ref{sub:the-internal-language-of-a-category}, when such categories possess suitable structure---for instance, when they are toposes, as all of the cited examples are---one can already engage in theory-building through their internal languages, without invoking the complex procedures that operate ``under the hood.'' 

For example, in a presheaf category $[\mathcal{C}^{\mathrm{op}}, \mathbf{Set}] = \mathcal{C}^@$---such as $\mathbf{Mod}^@$~\cite{mazzola2002topos} or $\mathbf{Rel}^@$~\cite{flieder2024towards}---a presheaf $A$ may be regarded as a type, so that writing $x \in A$ serves as shorthand for either:
\begin{itemize}
    \item $x : A$ in the internal language of $\mathcal{C}^@$; or
    \item when $A : X \to \Omega$ is a propositional function, $x \in A$ denotes the evaluation of the propositional function
    \[
        \in \; \colon X \times \Omega^X \longrightarrow \Omega
    \]
    at the pair $(x, A) : X \times \Omega^X$ (see Example~\ref{ex:interpretation-of-scale-membership}).
\end{itemize}

Lastly, in keeping with the idea of a \emph{lingua franca}, future developments will depend on how individual theorists make use of the framework to articulate their own mathemusical inquiries. The generality of the framework is what grants it this openness: it imposes no particular ontological commitments concerning what kinds of things exist apart from the minimal assumptions of terms, types, and their relations. Its syntax is highly general yet deliberately bare: agnostic with respect to semantics, but affording semantic interpretation through categorical semantics. In this way, syntactic constructions may be interpreted across a variety of ontological contexts, allowing the same formal language to serve many different theoretical contexts.

\pagebreak

\printbibliography

@book{awodey2010category,
  title={Category Theory},
  author={Awodey, Steve},
  year={2010},
  publisher={Oxford University Press},
  address={Oxford, UK},
}

@article{awodey1996structure,
  title     = {Structure in Mathematics and Logic: A Categorical Perspective},
  author    = {Awodey, Steve},
  journal   = {Philosophia Mathematica},
  volume    = {4},
  number    = {3},
  year      = {1996},
  publisher = {Oxford University Press}
}

@book{corfield2020modal,
  author    = {Corfield, David},
  title     = {Modal Homotopy Type Theory: The Prospect of a New Logic for Philosophy},
  publisher = {Oxford University Press},
  address   = {Oxford, UK},
  year      = {2020},
  %isbn      = {978-0198853404}
}

@article{flieder2024towards,
author = {Flieder,Drew},
title = {Towards a mathematical foundation for music theory and composition: a theory of structure},
journal = {Journal of Mathematics and Music},
year = {2024},
doi = {10.1080/17459737.2024.2379788},
}

@article{flieder2025lifting,
author = {Drew Flieder},
title = {Lifting pitch-class spaces into ambient temperaments},
journal = {Journal of Mathematics and Music},
year = {2025},
doi = {10.1080/17459737.2025.2496139},
}

@phdthesis{flieder2025meta,
  title={Toward a Universal Meta-Theory of Music: An Inquiry into Musical Signs},
  author={Flieder, Drew},
  year={2025},
  school={University of California, Santa Barbara}
}

@inproceedings{gambino2013polynomial,
  title={Polynomial functors and polynomial monads},
  author={Gambino, Nicola and Kock, Joachim},
  booktitle={Mathematical proceedings of the cambridge philosophical society},
  volume={154},
  number={1},
  %pages={153--192},
  year={2013},
  organization={Cambridge University Press}
}

@book{goldblatt1984topoi,
  author    = {Robert Goldblatt},
  title     = {Topoi: The Categorial Analysis of Logic},
  edition   = {Revised edition},
  year      = {1984},
  publisher = {Elsevier Science Publishers B.V.},
  address   = {Amsterdam},
  series    = {Studies in Logic and the Foundations of Mathematics},
  volume    = {98},
}

@book{hott2013,
  title     = {Homotopy Type Theory: Univalent Foundations of Mathematics},
  author    = {{The Univalent Foundations Program}},
  year      = {2013},
  publisher = {Institute for Advanced Study},
  address   = {Princeton, NJ, USA},
  %isbn      = {978-0-692-15790-0},
  url       = {https://homotopytypetheory.org/book/}
}

@book{jacobs1999categorical,
  author    = {Jacobs, Bart},
  title     = {Categorical Logic and Type Theory},
  publisher = {Elsevier Science B.V.},
  address   = {Amsterdam, The Netherlands},
  year      = {1999},
  series    = {Studies in Logic and the Foundations of Mathematics},
  volume    = {141},
  %isbn      = {0-444-50170-3},
  %pages     = {784}
}

@book{johnstone2002sketches,
  author    = {Johnstone, Peter T.},
  title     = {Sketches of an Elephant: A Topos Theory Compendium},
  publisher = {Oxford University Press},
  address   = {Oxford, UK},
  year      = {2002},
  note = {2 vols.},
  %series    = {Oxford Logic Guides},
  %volume    = {43},
  %isbn      = {978-0198534259},
  %pages     = {562}
}

@book{lewin2007generalized,
  author    = {Lewin, David},
  title     = {Generalized Musical Intervals and Transformations},
  publisher = {Oxford University Press},
  address   = {New York, NY, USA},
  year      = {2007/1987},
  %isbn      = {978-0-19-531713-8},
  note      = {Originally published in 1987 by Yale University Press}
}

@book{maclane2012sheaves,
  title={Sheaves in Geometry and Logic: A First Introduction to Topos Theory},
  author={Mac Lane, Saunders and Moerdijk, Ieke},
  year={2012},
  publisher={Springer Science \& Business Media},
  address= {New York, NY, USA}
}

@incollection{martinlof1979,
title = {Constructive Mathematics and Computer Programming},
editor = {L. Jonathan Cohen and Jerzy Łoś and Helmut Pfeiffer and Klaus-Peter Podewski},
series = {Studies in Logic and the Foundations of Mathematics},
publisher = {Elsevier},
volume = {104},
%pages = {153-175},
year = {1982},
booktitle = {Logic, Methodology and Philosophy of Science VI},
issn = {0049-237X},
doi = {https://doi.org/10.1016/S0049-237X(09)70189-2},
url = {https://www.sciencedirect.com/science/article/pii/S0049237X09701892},
  author    = {Martin-L{\"o}f, Per},
}

@incollection{martin1998intuitionistic,
  author    = {Martin-L{\"o}f, Per},
  title     = {An Intuitionistic Theory of Types},
  booktitle = {Twenty-Five Years of Constructive Type Theory},
  editor    = {Sambin, Giovanni and Smith, Jan M.},
  volume    = {36},
  series    = {Oxford Logic Guides},
  year      = {1998},
  publisher = {Oxford University Press},
  address   = {Oxford, UK},
  %isbn      = {978-0-19-850127-5}
}

@article{martin2021intuitionistic,
  author    = {Martin-L{\"o}f, Per},
  title     = {Intuitionistic Type Theory},
  subtitle  = {Notes by Giovanni Sambin of a series of lectures given in Padova, June 1980},
  note      = {Digital edition prepared by Giovanni Sambin, 2021},
  year      = {1984},
  publisher = {Bibliopolis}
}

@article{mazzola1997semiotics,
  title={Semiotics of Music},
  author={Mazzola, Guerino},
  journal={A Handbook on the Sign-Theoretic Foundations of Nature and Culture},
  volume={3},
  year={1997}
}

@book{mazzola2002topos,
  title={The Topos of Music: Geometric Logic of Concepts, Theory, and Performance},
  author={Mazzola, Guerino et al.},
  year={2002},
  publisher={Birkh\"auser Verlag},
  address={Basel, CH}
}

@book{mazzola2018topos,
  title={The Topos of Music III: Gestures},
  author={Mazzola, Guerino and Guitart, Ren{\'e} and Ho, Jocelyn and Lubet, Alex and Mannone, Maria and Rahaim, Matt and Thalmann, Florian},
  journal={Third volume of second edition of},
  year={2018},
  publisher={Springer},
  address ={Heidelberg, GER}
}

@article{moerdijk2000wellfounded,
  title={Wellfounded trees in categories},
  author={Moerdijk, Ieke and Palmgren, Erik},
  journal={Annals of Pure and Applied Logic},
  volume={104},
  number={1-3},
  %pages={189--218},
  year={2000},
  publisher={Elsevier}
}

@misc{nlab_codomain_fibration,
  author       = {{nLab authors}},
  title        = {Codomain Fibration},
  year         = {2024},
  howpublished = {nLab, \url{https://ncatlab.org/nlab/show/codomain+fibration}},
  note         = {Accessed: 16 December 2024}
}

@article{noll2005topos,
author = {Noll, Thomas},
year = {2005},
journal = {Grazer Mathematische Berichte},
editor = {H. Fripertinger, L. Reich},
month = {01},
%pages = {1-26},
title = {The Topos of Triads},
volume = {347}
}

@article{popoff2018relational,
  title={Relational poly-Klumpenhouwer networks for transformational and voice-leading analysis},
  author={Popoff, Alexandre and Andreatta, Moreno and Ehresmann, Andr{\'e}e},
  journal={Journal of Mathematics and Music},
  volume={12},
  number={1},
  %pages={35--55},
  year={2018},
  publisher={Taylor \& Francis}
}

@online{shulman2013set,
  author  = {Shulman, Michael},
  title   = {From Set Theory to Type Theory},
  year    = {2013},
  month   = jan,
  day     = {7},
  url     = {https://golem.ph.utexas.edu/category/2013/01/from_set_theory_to_type_theory.html},
  note    = {Blog post on \emph{The n-Category Café}}
}

@incollection{sundholm1986proof,
  title     = {Proof Theory and Meaning},
  author    = {Sundholm, Göran},
  booktitle = {Handbook of Philosophical Logic: Volume III: Alternatives in Classical Logic},
  editor    = {Gabbay, Dov and Guenthner, Franz},
  series    = {Synthese Library},
  volume    = {166},
  %pages     = {471--506},
  year      = {1986},
  publisher = {D. Reidel Publishing Company},
  address   = {Dordrecht, Netherlands},
  doi       = {10.1007/978-94-009-5203-4_8},
  url       = {https://link.springer.com/chapter/10.1007/978-94-009-5203-4_8}
}

@article{sundholm1989constructive,
  title     = {Constructive Generalized Quantifiers},
  author    = {Sundholm, Göran},
  journal   = {Synthese},
  volume    = {79},
  number    = {1},
  %pages     = {1--10},
  year      = {1989},
  doi       = {10.1007/BF00873254},
  url       = {https://link.springer.com/article/10.1007/BF00873254}
}

@book{tymoczko2010geometry,
  title={A Geometry of Music: Harmony and Counterpoint in the Extended Common Practice},
  author={Tymoczko, Dmitri},
  year={2010},
  publisher={Oxford University Press},
  address   = {New York, NY, USA},
}

@book{tymoczko2023tonality,
  title={Tonality: An Owner's Manual},
  author={Tymoczko, Dmitri},
  year={2023},
  publisher={Oxford University Press},
  address   = {New York, NY, USA},
}

@unpublished{tymoczko2025concept,
  author       = {Dmitri Tymoczko},
  title        = {The Concept of Musical Space},
  note         = {Manuscript under review},
  year         = {2025}
}

\appendix

\section{Categorical Semantics: A Deeper Look}
\label{sec:categorical-semantics-a-deeper-look}

In Section~\ref{sec:categorical-semantics}, we introduced the basic ideas of categorical semantics, which associates meaningful structural content with the abstract types, functions, and relations of a type language. However, a fully expressive type language does more than describe primitive types and their morphisms: it also accommodates the formation of \emph{compound propositions} through logical connectives and quantifiers. 

Our earlier discussion did not address the categorical semantics of these logical operations. In this appendix, we take up that task, showing how categorical language articulates the semantics of logical connectives and quantifiers, thereby endowing the type language with its full expressive power.

\subsection{Categorical Semantics of Connectives}
\label{sub:Categorical Semantics of Connectives}

We begin with the categorical semantics of the logical connectives. Recall that for a signature $\Sigma$ containing a type $\mathsf{Prop}$ of propositions, its interpretation in a topos $\mathcal{E}$ is given by the subobject classifier, so that
\[
\llbracket \mathsf{Prop} \rrbracket = \Omega.
\]
A propositional function $p : X \to \mathsf{Prop}$ in the type language is then interpreted as a morphism 
\[
\llbracket p \rrbracket : \llbracket X \rrbracket \longrightarrow \Omega
\]
in $\mathcal{E}$. This is crucial for what follows, as the categorical semantics of the connectives can be understood in terms of operations in $\mathcal{E}$ involving the subobject classifier $\Omega$.\footnote{See \cite[Chapter~6]{goldblatt1984topoi} for a detailed treatment.}

\subsubsection{Logical Connectives as Operations on the Subobject Classifier}
\label{subsub:logical-connectives-as-operations-on-the-subobject-classifer}

Recall that the type $\mathsf{Prop}$ has two distinguished terms, $\bot : \mathsf{Prop}$ for ``false'' and $\top : \mathsf{Prop}$ for ``true.'' In the semantic category $\mathcal{E}$, we denote the interpretations of these terms by the same labels $\bot$ and $\top$. Thus, $\bot : 1 \to \Omega$ picks out the bottom element of the subobject classifier (corresponding to ``always false''), while $\top : 1 \to \Omega$ picks out the top element (corresponding to ``always true'').

\paragraph{Negation on $\Omega$} 

The negation operator is defined as the unique morphism $\neg : \Omega \to \Omega$ that makes the following square a pullback:
\[\begin{tikzcd}
	1 && \Omega \\
	\\
	1 && \Omega
	\arrow["\bot", from=1-1, to=1-3]
	\arrow[from=1-1, to=3-1]
	\arrow["\neg", from=1-3, to=3-3]
	\arrow["\top"', from=3-1, to=3-3]
\end{tikzcd}\]

This pullback condition characterizes $\neg$ as the operation that exchanges the truth values in $\Omega$: it sends $\top$ to $\bot$ and $\bot$ to $\top$.

\paragraph{Conjunction on $\Omega$}

The conjunction operator $\land : \Omega \times \Omega \to \Omega$ is defined as the characteristic morphism of the product map $(\top, \top) : 1 \to \Omega \times \Omega$. This construction reflects the fact that conjunction corresponds to the \emph{meet} operation in the Heyting algebra structure of $\Omega$. For two truth values $u, v \in \Omega$, their conjunction $u \land v$ is the greatest lower bound of $u$ and $v$. Equivalently, when $\Omega$ is regarded as a posetal category, this meet coincides with the categorical product of $u$ and $v$.

\paragraph{Disjunction on $\Omega$}

To define the disjunction operator, we first define the morphism
    \[
    \begin{matrix}
    (\top, \mathrm{Id}_\Omega) : & \Omega & \longrightarrow & \Omega \times \Omega \\
    & v & \longmapsto & (\top, v)  \\
    \end{matrix}    
    \] 
    which sends a truth value $v$ in $\Omega$ to the pair $(\top, v)$. Then, the disjunction operator $\lor : \Omega \times \Omega \rightarrow \Omega$ is the character map of the morphism 
    \[ \left[ (\top, \mathrm{Id}_\Omega), \; (\mathrm{Id}_\Omega, \top) \right] : \Omega + \Omega \longrightarrow \Omega \times \Omega. \]
    This construction reflects that disjunction corresponds to the \emph{join} operation in the Heyting algebra structure of $\Omega$. Given two truth values $u, v \in \Omega$, their disjunction $u \lor v$ is the least upper bound of $u$ and $v$, capturing the idea that $u \lor v$ holds whenever either $u$ or $v$ does. From a categorical perspective, this join is dual to the meet: when $\Omega$ is viewed as a posetal category, it corresponds to the categorical coproduct rather than the product.
    
To make this construction more concrete, consider the case in $\mathbf{Set}$, where 
\[
\Omega = 2 = \{\bot, \top\}.
\] 
In this setting, the coproduct $2 + 2$ can be written as $\{\bot_1, \top_1, \bot_2, \top_2\}$, where the subscripts indicate whether an element comes from the first or second copy of $2$. The morphism 
\[
\left[ (\top, \mathrm{Id}_\Omega), \; (\mathrm{Id}_\Omega, \top) \right] : 2 + 2 \longrightarrow 2 \times 2
\]
is then defined by
\[
\begin{aligned}
\top_1 &\longmapsto (\top, \top) \\
\bot_1 &\longmapsto (\top, \bot) \\
\top_2 &\longmapsto (\top, \top) \\
\bot_2 &\longmapsto (\bot, \top)
\end{aligned}
\]
so that the image of this map consists of all pairs in $2 \times 2$ except $(\bot, \bot)$. Consequently, the characteristic morphism 
\[
\lor : 2 \times 2 \longrightarrow 2
\]
assigns $\top$ to every pair except $(\bot, \bot)$, which it sends to $\bot$, exactly as expected for logical disjunction.

\paragraph{Implication on $\Omega$}

The implication operator $\Rightarrow : \Omega \times \Omega \to \Omega$ is defined as the characteristic morphism of 
\[
e : \ \leq \ \longrightarrow \Omega \times \Omega,
\]
where $\leq$ is the equalizer of the following diagram:
\[\begin{tikzcd}
	{\Omega\times \Omega} && \Omega
	\arrow["\land", shift left=3, from=1-1, to=1-3]
	\arrow["{\mathrm{pr}_1}"', shift right=3, from=1-1, to=1-3]
\end{tikzcd}\]

To see how the subobject $\leq$ arises, consider where the morphisms $\land$ and $\mathrm{pr}_1$ agree. The morphism $\land : \Omega \times \Omega \to \Omega$ computes the meet (greatest lower bound) of two truth values, while the projection $\mathrm{pr}_1 : \Omega \times \Omega \to \Omega$ returns the first component. For any $(u, v) \in \Omega \times \Omega$, we have
\[
\mathrm{pr}_1(u, v) = u \quad \text{and} \quad \land(u, v) = u \wedge v.
\]
These are equal precisely when $u = u \wedge v$, i.e.\ when $u \leq v$ in the Heyting algebra order on $\Omega$. Thus, the equalizer $\leq$ is the subobject of $\Omega \times \Omega$ consisting of pairs $(u, v)$ such that $u \leq v$.

Categorically, the implication operator $\Rightarrow$ therefore classifies the subobject $\leq$, encoding the logical idea that “if $u$ is true, then $v$ is also true.” In other words, $u \Rightarrow v$ holds exactly when $u \leq v$ in the internal order of $\Omega$.

\subsubsection{Extending Connectives to Typed Propositions}
\label{subsub:extending-connectives-to-typed-propositions}

The operators on $\Omega$ discussed above provide the categorical semantics of the logical connectives. To extend these to typed propositions, consider an interpretation $M : \mathcal{C}\ell(\Sigma) \to \mathcal{E}$. For convenience, we abbreviate the interpretation of any type-theoretic expression $X$ under $M$, written $\llbracket X \rrbracket_M$, simply as $\llbracket X \rrbracket$.

\paragraph{Negation on Typed Propositions} 

For a relation $R$ of type $A_1, \ldots, A_n$, the categorical semantics of its negation $\neg R$ is given by the composition of morphisms:
\[\begin{tikzcd}
	{\llbracket A_1, \ldots, A_n\rrbracket} && \Omega \\
	\\
	\\
	\Omega
	\arrow["{\llbracket R \rrbracket}", from=1-1, to=1-3]
	\arrow["{\llbracket \neg R \rrbracket \ = \ \neg  \ \circ \  \llbracket R \rrbracket }"', from=1-1, to=4-1]
	\arrow["\neg", from=1-3, to=4-1]
\end{tikzcd}\]

\paragraph{Conjunction on Typed Propositions}

For relations $R$ of type $A_1, \ldots, A_n$ and $S$ of type $B_1, \ldots, B_m$, the categorical semantics of its conjunction $R \land S$ is given by the composition of morphisms:
\[\begin{tikzcd}
	{\llbracket A_1, \ldots, A_n\rrbracket \times \llbracket B_1, \ldots, B_m \rrbracket} && {\Omega \times \Omega} \\
	\\
	\\
	\Omega
	\arrow["{\langle \llbracket R \rrbracket,  \  \llbracket S \rrbracket \rangle }", from=1-1, to=1-3]
	\arrow["{\llbracket R \ \land \  S \rrbracket  \ = \  \land \  \circ \  \langle \llbracket R \rrbracket,  \  \llbracket S \rrbracket \rangle}"', from=1-1, to=4-1]
	\arrow["\land", from=1-3, to=4-1]
\end{tikzcd}\]

\paragraph{Disjunction on Typed Propositions} 

For relations $R$ of type $A_1, \ldots, A_n$ and $S$ of type $B_1, \ldots, B_m$, the categorical semantics of its disjunction $R \lor S$ is given by the composition of morphisms:
\[\begin{tikzcd}
	{\llbracket A_1, \ldots, A_n\rrbracket \times \llbracket B_1, \ldots, B_m \rrbracket} && {\Omega \times \Omega} \\
	\\
	\\
	\Omega
	\arrow["{\langle \llbracket R \rrbracket,  \  \llbracket S \rrbracket \rangle }", from=1-1, to=1-3]
	\arrow["{\llbracket R \ \lor \  S \rrbracket  \ = \  \lor \  \circ \  \langle \llbracket R \rrbracket,  \  \llbracket S \rrbracket \rangle}"', from=1-1, to=4-1]
	\arrow["\lor", from=1-3, to=4-1]
\end{tikzcd}\]

\paragraph{Implication on Typed Propositions} 

For relations $R$ of type $A_1, \ldots, A_n$ and $S$ of type $B_1, \ldots, B_m$, the categorical semantics of its implication $R \Rightarrow S$ is given by the composition of morphisms:
\[\begin{tikzcd}
	{\llbracket A_1, \ldots, A_n\rrbracket \times \llbracket B_1, \ldots, B_m \rrbracket} && {\Omega \times \Omega} \\
	\\
	\\
	\Omega
	\arrow["{\langle \llbracket R \rrbracket,  \  \llbracket S \rrbracket \rangle }", from=1-1, to=1-3]
	\arrow["{\llbracket R \ \Rightarrow \  S \rrbracket  \ = \  \Rightarrow \  \circ \  \langle \llbracket R \rrbracket,  \  \llbracket S \rrbracket \rangle}"', from=1-1, to=4-1]
	\arrow["\Rightarrow", from=1-3, to=4-1]
\end{tikzcd}\]

\subsection{Categorical Semantics of Quantifiers}
\label{sub:categorical-semantics-of-quantifiers}

Given a context 
\[
\Gamma = x_1 : A_1, \ldots, x_n : A_n
\]
and a propositional function 
\[
\Gamma, \; y : Y \; \vdash \; \phi : \mathsf{Prop},
\]
we may form a new proposition in the context $\Gamma$ by quantifying over $y$. For example, universal quantification yields
\[
\Gamma  \;  \vdash  \;  \forall y : Y.\, \phi : \mathsf{Prop},
\]
while existential quantification yields
\[
\Gamma  \;  \vdash  \;  \exists y : Y.\, \phi : \mathsf{Prop}.
\]

To formalize the categorical semantics of the quantifiers, we express them in terms of adjunctions.\footnote{The exposition here follows \cite{awodey1996structure}.}  Consider the hom-set $\mathcal{E}[X, \Omega]$, which represents the set of propositional functions on $X$ and carries the Heyting algebra structure of subobjects of $X$. For the operation 
\[
x^* : \mathcal{E}[1, \Omega] \longrightarrow \mathcal{E}[X, \Omega]
\] 
that sends each truth value $p : 1 \rightarrow \Omega$ to the constant $p$-valued propositional function on $X$, the quantifiers arise as the right and left adjoints to $x^*$. We now make this construction explicit.

\subsubsection{The Universal Quantifier}

First, define a second-order propositional function
\[
\begin{matrix}
\forall_X : & \Omega^X & \longrightarrow & \Omega \\
& \phi & \longmapsto & \bigwedge_{x : X} \phi(x)
\end{matrix}
\]
which maps each propositional function $\phi : X \to \Omega$ to the meet of the truth values that $\phi$ assigns to elements of $X$. Intuitively, this expresses the idea that $\forall_X(\phi)$ is \emph{true} precisely when $\phi(x)$ holds for \emph{every} $x : X$. Since $\forall_X$ takes propositional functions as input and returns a truth value, it is indeed a second-order propositional function.

Given a propositional function $\phi : X \rightarrow \Omega$, its corresponding morphism 
\[
\lambda_x.\phi : 1 \longrightarrow \Omega^X
\]
is the unique morphism that classifies $\phi$ as an element of the exponential $\Omega^X$. By composing this morphism with 
\[
\forall_X : \Omega^X \longrightarrow \Omega,
\]
we obtain a truth value in $\Omega$ representing the result of universally quantifying $\phi$ over $X$. In other words, $\forall_X \circ \lambda_x.\phi$ is the semantic interpretation of the formula $\forall x : X.\,\phi(x)$.\footnote{Here we abuse notation slightly by writing $X$ and $\phi$ both for the syntactic context and formula, and for their corresponding semantic interpretations as objects and morphisms in the topos $\mathcal{E}$.} This construction is summarized in the following commutative diagram:
\[\begin{tikzcd}
	1 &&& {\Omega^X} \\
	\\
	\\
	&&& \Omega
	\arrow["{\lambda_x.\phi}", from=1-1, to=1-4]
	\arrow["{\forall_x.\phi \ = \ \forall_X \ \circ \ \lambda_x.\phi}"', from=1-1, to=4-4]
	\arrow["{\forall_X}", from=1-4, to=4-4]
\end{tikzcd}\]

Now, as we hinted at above, $\forall_x$ acts as a functor
\[ 
\forall_x : \mathcal{E}[X, \Omega]\cong \mathcal{E}[1, \Omega^X] \longrightarrow \mathcal{E}[1, \Omega] 
\]
which takes propositional functions on $X$ to truth values in $\Omega$. This functorial action reflects the process of universal quantification over $X$, mapping each propositional function $\phi : X \to \Omega$ to the truth value $\forall x : X.\; \phi(x)$ in $\Omega$.

The right adjointness of $\forall_x$ to $x^*$ is expressed by the following Galois connection. For any truth value $p : 1 \rightarrow \Omega$ and any propositional function $\phi : X \rightarrow \Omega$, we have
\[ 
x^*.p \leq \phi \quad \text{if and only if} \quad p \leq \forall_x. \phi, 
\]
meaning that $\forall_x$ is right adjoint to $x^*$. 

\subsubsection{The Existential Quantifier}

On the other hand, the existential quantifier arises as the left adjoint of the  functor $x^*$. Specifically, we define the functor
\[ 
\exists_x : \mathcal{E}[X, \Omega] \longrightarrow \mathcal{E}[1, \Omega] 
\]
which maps propositional functions on $X$ to truth values in $\Omega$. To understand its action, we first introduce the second-order propositional function
\[
\begin{matrix}
	\exists_X : & \Omega^X & \longrightarrow & \Omega \\
	& \phi & \longmapsto & \bigvee_{x : X} \phi(x)
\end{matrix}	
\]
which sends a propositional function $\phi$ over $X$ to the join of the truth values that $\phi$ assigns to elements of $X$. This definition captures the usual intuition: the statement $\exists x.\,\phi(x)$ is \emph{true} just in case $\phi(x)$ holds for at least one element of $X$.

Given a propositional function $\phi : X \to \Omega$, we may view $\phi$ as a morphism 
\[
\lambda_x. \phi : 1 \longrightarrow \Omega^X.
\]
 The existential quantifier $\exists_x$ then acts on $\phi$ by the composite
\[ 
\exists_x . \phi = \exists_X \circ \lambda_x . \phi, 
\]
so that $\exists_x$ takes $\phi$ and maps it to a truth value in $\Omega$, as indicated in the following diagram:
\[\begin{tikzcd}
	1 &&& {\Omega^X} \\
	\\
	\\
	&&& \Omega
	\arrow["{\lambda_x.\phi}", from=1-1, to=1-4]
	\arrow["{\exists_x.\phi \ = \ \exists_X \ \circ \ \lambda_x.\phi}"', from=1-1, to=4-4]
	\arrow["{\exists_X}", from=1-4, to=4-4]
\end{tikzcd}\]

The adjoint relationship between $\exists_x$ and $x^*$ is expressed by the following Galois connection: for any truth value $q : 1 \to \Omega$ and any propositional function $\phi : X \to \Omega$, we have
\[
\exists_x.\phi \leq q 
\quad \text{if and only if} \quad 
\phi \leq x^*.q.
\]

\subsubsection{The General Case of Quantification}
\label{subsub:the-general-case-of-quantification}

Lastly, let us turn to the general case of a propositional function in several variables. Let 
\[
\phi : X \times Y_1 \times \cdots \times Y_n \longrightarrow \Omega
\]
be a propositional function. Set $ Y = Y_1 \times \cdots \times Y_n $, so that we may view $\phi$ as a propositional function 
\[
\phi : X \times Y \longrightarrow \Omega.
\]
The quantifiers $\forall_x$ and $\exists_x$ are again obtained by composing the classifying morphism $\lambda_x.\phi$ with $\forall_X$ and $\exists_X$, respectively:
\[
\forall_x.\phi = \forall_X \circ \lambda_x.\phi
\quad \text{and} \quad
\exists_x.\phi = \exists_X \circ \lambda_x.\phi.
\]

Now let $\psi : Y \to \Omega$ be a propositional function. To obtain a propositional function on $X \times Y$, we can extend the context by introducing a dummy variable of type $X$ and composing $\psi$ with the projection 
\[
\pi : X \times Y \longrightarrow Y,
\]
as shown in the diagram below:
\[\begin{tikzcd}
	{X \times Y} &&& Y \\
	\\
	\\
	&&& \Omega
	\arrow["\pi", from=1-1, to=1-4]
	\arrow["{\pi^*.\psi \ = \ \psi \ \circ \ \pi}"', from=1-1, to=4-4]
	\arrow["\psi", from=1-4, to=4-4]
\end{tikzcd}\]

This operation defines a functor 
\[
\pi^* : \mathcal{E}[Y, \Omega] \longrightarrow \mathcal{E}[X \times Y, \Omega]
\]
which maps propositional functions on $Y$ to propositional functions on $X \times Y$ by precomposing them with $\pi$. In the special case where $Y = 1$, this functor reduces to $ \pi^* = x^* $, the functor that extends propositional functions from the empty context to the context $ x : X $.

The general existential and universal quantifiers are now defined as functors
\[
\exists_x, \forall_x : \mathcal{E}[X \times Y, \Omega] \longrightarrow \mathcal{E}[Y, \Omega],
\]
and they serve as the left and right adjoints, respectively, to the functor $\pi^*$:
\[
\exists_x \dashv \pi^* \dashv \forall_x.
\]

\section{Propositions as Types}
\label{sec:propositions-as-types}

In this section, we explain the paradigm of \emph{propositions as types}, according to which propositions are associated not with truth values but with types. The central idea is that, rather than assigning a truth value to the proposition obtained by evaluating a predicate on a term, we interpret that proposition as the \emph{type of its proofs}: to prove it is to construct a term inhabiting the corresponding type.

To make this shift precise, we introduce the notions of \emph{fibrations} and \emph{slice categories}, which provide the categorical machinery needed to formalize a predicate as a family of types varying over the terms in a context. This framework is not only philosophically illuminating, but also practically useful, as it underlies the implementation of \emph{proof assistants}: computer programs that assist in constructing and verifying proofs. For these reasons, we present it here as a technical appendix that complements the main text.

\subsection{Fibrations}
\label{sub:fibrations}

We begin by introducing the notion of \emph{fibrations} and their relation to the logical structures discussed earlier. Throughout the main text, we alternated between two equivalent ways of representing propositional functions: as morphisms $\phi : X \to \mathsf{Prop}$ and as subobjects $\phi \hookrightarrow X$. We now make these notions more precise by situating them within the framework of fibrations and fibered categories.

A \emph{fibration} is a functor 
\[
p : \mathbb{E} \longrightarrow \mathbb{B}
\]
from a category $\mathbb{E}$, called the \emph{total category}, to a category $\mathbb{B}$, called the \emph{base category}. The base category $\mathbb{B}$ represents the possible \emph{contexts} or \emph{universes of discourse}, while the fiber $p^{-1}(B)$ over an object $B \in \mathbb{B}$ is the category of predicates defined over that context.

A classic example of a fibration is given by the \emph{predicate fibration}. Let $\mathbb{B} = \mathbf{Set}$ be the category of sets, and let $\mathbb{E} = \mathbf{Pred}$ be the category whose objects are pairs $(I, X)$, where $I$ is a set and $X \subseteq I$. In this setting, $X$ represents a predicate on $I$: we write $X(i)$ for $i \in X$ to emphasize that the elements $i \in I$ are those with respect to which the predicate $X$ is evaluated.

A morphism $(I, X) \to (J, Y)$ in $\mathbf{Pred}$ is a function $u : I \to J$ such that $X(i) \Rightarrow Y(u(i))$ for each $i \in I$. In other words, $u$ preserves the truth of predicates under reindexing. The forgetful functor 
\[
p : \mathbf{Pred} \longrightarrow \mathbf{Set}
\]
defined by $p(I, X) = I$ on objects and $p(u) = u$ on morphisms is a fibration. The fiber over a set $I$ is therefore the poset of subsets of $I$, ordered by inclusion, which we can understand as the category of predicates defined over the context $I$.\footnote{See \cite[11]{jacobs1999categorical}.}

For present purposes, we focus on the basic intuition behind a fibration. (For a formal definition, see \cite[27]{jacobs1999categorical}.) The essential idea is that a type theory over a signature $\Sigma$ can be interpreted in a fibration $p : \mathbb{E} \to \mathbb{B}$. This interpretation is carried out relative to a model $M$ of the classifying category $\mathcal{C}\ell(\Sigma)$ of the type theory, as illustrated below:
\[\begin{tikzcd}
	&& {\mathbb{E}} \\
	\\
	{\mathcal{C}\ell(\Sigma)} && {\mathbb{B}}
	\arrow["p", from=1-3, to=3-3]
	\arrow["M"', from=3-1, to=3-3]
\end{tikzcd}\]
In this setup, the model $M$ interprets the syntax of the type theory within the base category $\mathbb{B}$, while the fibration $p$ organizes the semantic layer above it: it classifies the predicates defined over the interpreted types.

Let us now examine an example illustrating the interpretation of a signature $\Sigma$ in the predicate fibration $p : \mathbf{Pred} \to \mathbf{Set}$. A relation $R$ on types $A_1, \ldots, A_n$ in the signature $\Sigma$ corresponds to an object in $\mathbf{Pred}$ above the set $\llbracket A_1, \ldots, A_n \rrbracket_M$, i.e., a predicate defined over the semantic domain of the relation.

To make this correspondence explicit, we first define a fibration 
\[
p' : \mathcal{L}(\Sigma, \Pi) \longrightarrow \mathcal{C}\ell(\Sigma),
\] 
where $\mathcal{C}\ell(\Sigma)$ denotes the classifying category of the type theory, and $\mathcal{L}(\Sigma, \Pi)$ is the category whose objects are propositional functions over the contexts of $\mathcal{C}\ell(\Sigma)$, defined as follows. For each object (that is, each context)
\[
\Gamma = x_1 : A_1, \ldots, x_n : A_n
\]
in $\mathcal{C}\ell(\Sigma)$, the fiber of $p'$ over $\Gamma$ is the collection of propositional functions definable in that context:
\[
\Gamma \; \vdash \; \varphi : \mathsf{Prop}.
\]

As with the predicate fibration, morphisms between propositional functions are induced by morphisms $f : \Gamma \to \Delta$ between contexts. Given propositional functions
\[
\Gamma \; \vdash \; \varphi : \mathsf{Prop}
\quad\text{and}\quad
\Delta \; \vdash \; \psi : \mathsf{Prop},
\]
a morphism
\[
(\Gamma \; \vdash \; \varphi : \mathsf{Prop}) \longrightarrow (\Delta \; \vdash \; \psi : \mathsf{Prop})
\]
is a context morphism $f : \Gamma \to \Delta$ such that
\[
\varphi(x) \Rightarrow \psi(f(x)) \quad \text{for all } x : \Gamma.
\]
In the special case where $\Gamma = \Delta$, such morphisms represent entailments between propositions in the same context.

Hence, we have a morphism of fibrations:
\[\begin{tikzcd}
	{\mathcal{L}(\Sigma, \Pi)} && {\mathbf{Pred}} \\
	\\
	{\mathcal{C}\ell(\Sigma)} && {\mathbf{Set}}
	\arrow["{M'}", from=1-1, to=1-3]
	\arrow["{p'}"', from=1-1, to=3-1]
	\arrow["p", from=1-3, to=3-3]
	\arrow["M"', from=3-1, to=3-3]
\end{tikzcd}\]
where $M'$ interprets propositional functions in the context $\Gamma$ as objects in the fiber of $p$ over $\llbracket \Gamma \rrbracket_M$.

\subsection{From Truth Values to Types}
\label{sub:from-truth-values-to-types}

We may call the framework in which a proposition is understood as the evaluation of a predicate $\varphi$ on a term $x : X$ to yield a value $\varphi(x) : \mathsf{Prop}$ the \emph{propositions-as-truth-values} paradigm, which is the conventional perspective. In contrast, in the \emph{propositions-as-types} paradigm, the meaning of a proposition is not merely a truth value, but the \emph{type of its proofs}.

For example, rather than interpreting the proposition ``971 is not prime'' as a statement that is either true or false, we interpret it as a type: namely, the type consisting of all possible proofs of that proposition. And what constitutes a proof that a number $x$ is not prime? It is the construction of a pair of integers $n$ and $m$, both greater than $1$, together with a demonstration that their product equals $x$.\footnote{See \cite{martin1998intuitionistic}.}

Accordingly, instead of expressing a proposition in context as
\[
x : X \; \vdash \; \varphi(x) : \mathsf{Prop},
\]
we now write
\[
x : X \; \vdash \; \varphi(x) : \mathsf{Type}.
\]
Here the predicate $\varphi$, when evaluated at a term $x$, yields a type: specifically, the type whose inhabitants are the proofs of the proposition. To prove $\varphi(x)$ is therefore to \emph{construct a term} of that type. 

Type theories that allow for the construction of types that depend on terms of other types—so-called \emph{dependent types}—are known as \emph{dependent type theories}. A typical example of a dependent type is a family
\[
x : X \; \vdash \; B(x) : \mathsf{Type},
\]
where the type $B(x)$ varies with the choice of term $x : X$.

In the familiar \emph{propositions-as-truth-values} setting, the categorical semantics of a type theory is given by a \emph{subobject fibration}  
\[
p : \mathbf{Sub}(\mathbb{B}) \longrightarrow \mathbb{B},
\]  
where $\mathbb{B}$ is a topos. Here the fiber over an object $B$ consists of its subobjects---equivalently, monomorphisms into $B$---which correspond to predicates as morphisms $ B \to \Omega$, with $\Omega$ the subobject classifier. This expresses the idea that evaluating a predicate at an element of $B$ yields a truth value in $\Omega$. The predicate fibration discussed earlier is the special case of this construction obtained by taking $\mathbb{B} = \mathbf{Set}$.

However, once we adopt the propositions-as-types paradigm, this picture no longer suffices. Evaluating a predicate at a term no longer produces a mere truth value, but rather a type, potentially with many distinct inhabitants, i.e., many possible proofs. Consequently, predicates are no longer adequately modeled as subobjects. 

Thus, rather than a predicate $\varphi$ on $X$ being interpreted as a subobject 
\[
\llbracket \varphi \rrbracket \hooklongrightarrow \llbracket X \rrbracket,
\]
it is now interpreted as a morphism 
\[
q : \llbracket \varphi \rrbracket \longrightarrow \llbracket X \rrbracket
\]
whose fiber over each element $\llbracket x \rrbracket \in \llbracket X \rrbracket$ is the type of proofs of $\varphi(x)$.

\subsection{Categorical Semantics of Dependent Type Theory}
\label{sub:categorical-semantics-of-dependent-type-theory}

One approach to the categorical semantics of dependent type theory proceeds via the \emph{codomain fibration}. The arrow category $\mathbb{B}^\rightarrow$ of a category $\mathbb{B}$ has as its objects the morphisms $f : X \to X'$ in $\mathbb{B}$, and as its morphisms the commuting squares of the form:
\[\begin{tikzcd}
	X && Y \\
	\\
	{X'} && {Y'}
	\arrow["u", from=1-1, to=1-3]
	\arrow["f"', from=1-1, to=3-1]
	\arrow["g", from=1-3, to=3-3]
	\arrow["v"', from=3-1, to=3-3]
\end{tikzcd}\]
The \emph{codomain fibration} 
\[
\mathrm{cod} : \mathbb{B}^\rightarrow \longrightarrow \mathbb{B}
\]
is defined by sending each arrow $f : X \to X'$ to its codomain $X'$, and each commuting square as above to the bottom arrow $v : X' \to Y'$.\footnote{See \cite{nlab_codomain_fibration}.}

The fiber of $\mathrm{cod}$ over an object $X$ is the slice category $\mathbb{B}/X$. An object in this slice category is a morphism $\psi : A \to X$, which can be understood as a dependent type indexed by $X$: the fiber $\psi^{-1}(x)$ over an element $x \in X$ is the object of proofs of the proposition $\psi(x)$.

To summarize, in dependent type theory a proposition in context $z : Z$ is expressed as a judgment
\[
z : Z \;\vdash\; \varphi(z) : \mathsf{Type}.
\]
This states that, when evaluated at a term $z : Z$, the predicate $\varphi$ yields a type, namely  the type of its proofs. Semantically, such a judgment is interpreted as an object
\[
q : \llbracket \varphi \rrbracket \longrightarrow \llbracket Z \rrbracket
\]
in the slice category $\mathbb{B}/\llbracket Z \rrbracket$, where the fiber over $\llbracket z \rrbracket \in \llbracket Z \rrbracket$ is the type of proofs of $\varphi(z)$.

In what follows, we develop the categorical semantics of logical connectives and quantifiers for dependent type theory, interpreting them in codomain fibrations.

\subsubsection{Connectives in Dependent Type Theory}
\label{subsub:connectives-in-dependent-type-theory}

The logical connectives in dependent type theory are interpreted by categorical constructions in the slice category $\mathbb{B}/\llbracket \Gamma \rrbracket$, where $\Gamma$ is a context in the type theory. Given two predicates $\phi$ and $\psi$ in context $\Gamma$, their semantic interpretations are as follows: 

\begin{itemize}
    \item \textbf{Conjunction} ($\phi \land \psi$): interpreted as the product 
    \[
    \llbracket \phi \land \psi \rrbracket \;\cong\; \llbracket \phi \rrbracket \times \llbracket \psi \rrbracket.
    \]
    \textit{Explanation:} A proof of $\phi \land \psi$ consists of both a proof of $\phi$ and a proof of $\psi$. The categorical product captures this by requiring a pair of elements, one from each type.

    \item \textbf{Disjunction} ($\phi \lor \psi$): interpreted as the coproduct 
    \[
    \llbracket \phi \lor \psi \rrbracket \;\cong\; \llbracket \phi \rrbracket + \llbracket \psi \rrbracket.
    \]
    \textit{Explanation:} A proof of $\phi \lor \psi$ is either a proof of $\phi$ or a proof of $\psi$. The coproduct expresses this disjunctive structure: elements of the sum type are tagged as belonging to one disjunct or the other.

    \item \textbf{Implication} ($\phi \Rightarrow \psi$): interpreted as the internal hom 
    \[
    \llbracket \phi \Rightarrow \psi \rrbracket \;\cong\; \llbracket \psi \rrbracket^{\llbracket \phi \rrbracket}.
    \]
    \textit{Explanation:} By the Curry–Howard correspondence, a proof of $\phi \Rightarrow \psi$ is a function that transforms any proof of $\phi$ into a proof of $\psi$. The internal hom captures this functional behavior.

    \item \textbf{Negation} ($\neg \phi$): interpreted as 
    \[
    \llbracket \neg \phi \rrbracket \;\cong\; 0^{\llbracket \phi \rrbracket},
    \]
    where $0$ is the initial object of $\mathbb{B}/\llbracket \Gamma \rrbracket$.  
    \textit{Explanation:} Since no morphisms exist from a non-initial object into $0$, the type $0^{\llbracket \phi \rrbracket}$ is inhabited if and only if $\llbracket \phi \rrbracket$ is initial. Thus, $\neg \phi$ is provable precisely when $\phi$ is unprovable.

    \item \textbf{Truth} ($\top$): interpreted as the terminal object of the slice category, namely the identity morphism 
    \[
    \llbracket \top \rrbracket \;\cong\; \mathrm{id}_{\llbracket \Gamma \rrbracket} : \llbracket \Gamma \rrbracket \longrightarrow \llbracket \Gamma \rrbracket.
    \]

    \item \textbf{Falsehood} ($\bot$): interpreted as the initial object of the slice category, 
    \[
    \llbracket \bot \rrbracket \;\cong\; 0 \longrightarrow \llbracket \Gamma \rrbracket,
    \]
\end{itemize}

\subsubsection{Quantifiers in Dependent Type Theory}
\label{subsub:quantifiers-in-dependent-type-theory}

We now turn to the categorical semantics of quantifiers within the propositions-as-types framework. In what follows, we drop the notational distinction between type-theoretic judgments and their categorical semantics, writing expressions such as $x : X$ which may denote either a type-theoretic context or its interpretation in $\mathbb{B}$. This simplifies the exposition, though it is formally an abuse of notation. 

Consider a projection map 
\[
\pi : X \times Y \longrightarrow Y
\]
in the base category $\mathbb{B}$ of a fibration $p : \mathbb{E} \to \mathbb{B}$. Such a projection induces a functor between slice categories,
\[
\pi^* : \mathbb{B}/Y \longrightarrow \mathbb{B}/(X \times Y),
\]
known as the \emph{base change} or \emph{reindexing} functor. This functor corresponds to the same construction as the ``dummy variable'' functor introduced earlier in the propositions-as-subobjects approach (see Section~\ref{subsub:the-general-case-of-quantification}).

The base change functor $\pi^*$ sends predicates $\phi : A \to Y$ in $\mathbb{B}/Y$ to their pullback along $\pi$, as represented by the following commutative square:
\[
\begin{tikzcd}
	{(X \times Y) \times_Y A} && A \\
	\\
	{X \times Y} && Y
	\arrow[from=1-1, to=1-3]
	\arrow[from=1-1, to=3-1]
	\arrow["\phi", from=1-3, to=3-3]
	\arrow["\pi"', from=3-1, to=3-3]
\end{tikzcd}
\]
Put simply, $\pi^*$ transforms a predicate over $Y$ into a predicate over $X \times Y$ by \emph{reindexing} it along $\pi$—that is, by interpreting it relative to a context extended with an additional dummy variable $x : X$.

As discussed in Section~\ref{subsub:the-general-case-of-quantification}, the base change functor $\pi^*$ admits both a left and a right adjoint:
\[
\Sigma_\pi \dashv \pi^* \dashv \Pi_\pi.
\]
These adjoints correspond, respectively, to existential and universal quantification in dependent type theory.

The definitions of these functors are as follows. The functor $\Sigma_\pi$ sends a predicate $\psi : B \to X \times Y$ to the predicate given by composition with $\pi$, i.e.\ $\pi \circ \psi : B \to Y$:
\[
\begin{tikzcd}
	B && {X \times Y} \\
	\\
	&& Y
	\arrow["\psi", from=1-1, to=1-3]
	\arrow["{\pi \ \circ \ \psi}"', from=1-1, to=3-3]
	\arrow["\pi", from=1-3, to=3-3]
\end{tikzcd}
\]

This functor corresponds to existential quantification. The categorical semantics of the typing judgment
\[
y : Y \;\vdash\; \exists x : X.\, \psi(x, y)  : \mathsf{Type}
\]
is the morphism $\pi \circ \psi : B \to Y$, whose fiber over a term $y : Y$ is the dependent pair type
\[
\sum_{x : X} \psi(x, y),
\]
consisting of pairs $(x, p)$ where $x : X$ and $p$ is a proof of $\psi(x, y)$. If this dependent sum type is \emph{inhabited}—that is, if there exists a term of it—then there exists an $x$ such that $\psi(x, y)$ holds, precisely capturing the semantics of existential quantification.

We now turn to the right adjoint $ \Pi_\pi $. This functor sends a predicate  
\[
\psi : B \longrightarrow X \times Y
\]
to a new predicate  
\[
\Pi_\pi(\psi) : \Pi_\pi B \longrightarrow Y
\]
whose fiber over a term $ y : Y $ is the dependent product
\[
\prod_{x : X} \psi(x, y),
\]
the type of families of proofs $(p_x)_{x \in X}$, one for each $x$, with $p_x$ a proof of $\psi(x, y)$. 

This construction captures the semantics of universal quantification. The interpretation of the typing judgment
\[
y : Y \;\vdash\; \forall x : X.\, \psi(x, y) : \mathsf{Type}
\]
is precisely the morphism $ \Pi_\pi(\psi) : \Pi_\pi B \to Y $, whose fiber over each $ y : Y $ is the dependent product above. If this fiber is inhabited—i.e.\ if a term of this dependent product type exists—it provides a proof of $ \psi(x, y) $ for \emph{every} $ x : X $, precisely as required by universal quantification.

The dependent product can also be understood as the object of sections of the bundle 
\[
b : \sum_{x : X} \psi(x, y) \longrightarrow X,
\]
where each element of the fiber over $x$ is a proof of $ \psi(x, y) $. Concretely, a section 
\[
s : X \longrightarrow \sum_{x : X} \psi(x, y)
\]
is a morphism such that the following diagram commutes:
\[
\begin{tikzcd}
	X && {\sum\limits_{x : X} \psi(x,y)} \\
	\\
	&& X
	\arrow["s", from=1-1, to=1-3]
	\arrow["{\mathrm{Id}_X}"', from=1-1, to=3-3]
	\arrow["b", from=1-3, to=3-3]
\end{tikzcd}
\]
Such a section assigns to each $x : X$ a proof of $\psi(x, y)$, thereby exhibiting the dependent product as the object of global sections of this family. In particular, the type of global sections 
\[
\Gamma\!\left(\sum_{x : X} \psi(x, y)\right)
\]
is canonically isomorphic to the dependent product:
\[
\prod_{x : X} \psi(x, y) \;\cong\; \Gamma\!\left(\sum_{x : X} \psi(x, y)\right).
\]

We close this appendix with some examples. 

\begin{example}[Proof that a number is not prime]
As discussed in Section~\ref{sub:from-truth-values-to-types}, the proposition ``$x$ is not prime'' can be proved by exhibiting natural numbers $n$ and $m$, each greater than $1$, such that $nm = x$. Accordingly, the type corresponding to the proposition $\mathsf{notPrime}(x)$ is
\[
\mathsf{notPrime}(x) = \sum_{n : \mathbb{N}_{>1}} \sum_{m : \mathbb{N}_{>1}} (\, nm = x \,).
\]
A term of this type is a triple $(n, m, r)$ where $n, m : \mathbb{N}_{>1}$ and $r$ is a proof that $nm = x$. Thus, the type $\mathsf{notPrime}(x)$ classifies all possible proofs of the proposition ``$x$ is not prime,'' rather than merely assigning a truth value.
\end{example}

\begin{example}[Proof that a pitch-class set is all-interval]
\label{ex:proof-all-interval}
For a more musical example, consider the proposition that a pitch-class set $C$ is \emph{all-interval}, i.e.\ that it contains every possible interval class in a pitch-class universe $\mathsf{PC}$.

We begin by endowing the pitch-class type $\mathsf{PC}$ with the structure of a finite cyclic group of order $n$ under mod-$n$ addition:
\[
+ : \mathsf{PC} \times \mathsf{PC} \longrightarrow \mathsf{PC}.
\]
This operation admits an additive identity $0$ and inverses $-x$ for each $x : \mathsf{PC}$. In standard pitch-class theory, this simply reflects the fact that pitch classes form a copy of the cyclic group $\mathbb{Z}_n$.

Given this structure, we define the \emph{pitch-class interval} function
\[
\begin{matrix}
\mathsf{pcint} : & \mathsf{PC} \times \mathsf{PC} &  \longrightarrow &  \mathsf{IVLS} \\
& (x, y) & \longmapsto & y - x
\end{matrix}
\]
where $\mathsf{IVLS}$ denotes the type of pitch-class intervals. 

Because inversion identifies an interval with its additive inverse, we pass to the quotient type of \emph{interval classes} $\mathsf{IC}$, with
\[
\mathsf{intclass} : \mathsf{IVLS} \longrightarrow \mathsf{IC}
\]
the canonical epimorphism. The type $\mathsf{IC}$ contains $\lfloor n/2 \rfloor + 1$ elements, including the trivial class $0$.

A pitch-class set is a subset of $\mathsf{PC}$, hence the type of all such sets is
\[
\mathsf{PCSets} \coloneqq P\mathsf{PC},
\]
where $P$ is the power-type operator.\footnote{Formally, $P\mathsf{PC}$ is equivalent to the function type $\mathsf{PC} \to \mathsf{Prop}$, whose elements correspond to predicates (or subtypes) over $\mathsf{PC}$.} Each element $C : \mathsf{PCSets}$ may therefore be regarded as a subtype of $\mathsf{PC}$.

To prove that a chord $C : \mathsf{PCSets}$ is all-interval is to prove that for every interval class $i : \mathsf{IC}$, there exist pitch classes $x, y : C$ such that $\mathsf{intclass}(\mathsf{pcint}(x, y)) = i$. In type-theoretic terms, a proof is therefore a term of the dependent type
\[
\mathsf{allInterval}(C) \coloneqq \prod_{i : \mathsf{IC}} \sum_{(x, y) : C \times C} \mathsf{intclass}(\mathsf{pcint}(x, y)) = i.
\]

The inner dependent sum
\[
\sum_{(x, y) : C \times C} \mathsf{intclass}(\mathsf{pcint}(x, y)) = i
\]
consists of all pairs $((x, y), p_i)$, where $(x, y) : C \times C$ and $p_i$ is a proof that $\mathsf{intclass}(\mathsf{pcint}(x, y)) = i$. The dependent product then ranges over all interval classes. If, for some $i$, this sum type is empty, the dependent product is uninhabited, and hence no proof of $\mathsf{allInterval}(C)$ exists—meaning that $C$ is not all-interval. 

The general predicate expressing that a chord is all-interval can thus be written as the typing judgment
\[
C : \mathsf{PCSets} \;\vdash\; \mathsf{allInterval}(C) : \mathsf{Type},
\]
whose semantic interpretation is a morphism 
\[
\llbracket \mathsf{allInterval} \rrbracket \longrightarrow \llbracket \mathsf{PCSets} \rrbracket
\]
in the slice category $\mathbb{B}/\llbracket \mathsf{PCSets} \rrbracket$.

\end{example}

\begin{example}[Proof that a dominant-function chord contains the leading tone]
\label{ex:proof-domfunc-leading-tone}
This example illustrates the versatility of the dependent type-theoretic framework by showing how dependent types may themselves vary over terms of other dependent types. We construct the type corresponding to the proposition that, for a given key $k : \mathsf{Key}$, any chord $c : \mathsf{Chord}$ with dominant function in $k$ contains the leading tone of $k$:\footnote{This example has a similar syntactic form to the canonical ``donkey sentences,'' such as ``John beats every donkey he owns''; see \cite{corfield2020modal}, \cite[597]{jacobs1999categorical}, \cite{sundholm1989constructive}, and \cite{sundholm1986proof}.}
\[
\prod_{x : \sum_{c : \mathsf{Chord}} \mathsf{domfunc}(c, k)} 
   \mathsf{contains}\!\left(\mathsf{pr}_1(x), \mathsf{lt}(k)\right)
\]

Let us unpack this expression. The dependent sum
\[
\sum_{c : \mathsf{Chord}} \mathsf{domfunc}(c, k)
\]
consists of pairs $(c, p)$ where $c$ is a chord and $p$ is a proof that $c$ has dominant function in $k$. The dependent product then quantifies over all such pairs, forming the type of proofs that the chord $c = \mathsf{pr}_1(x)$ contains the leading tone $\mathsf{lt}(k)$ of $k$.

If this dependent product type is inhabited (i.e.\ non-empty), then the proposition holds: every chord with dominant function in $k$ indeed contains the leading tone of $k$, as is the case in tonal harmony.

Now consider a broader class of scales $\mathsf{Scale}$, some of which contain no chords with dominant function. Let $s : \mathsf{Scale}$ be one such scale. In such a case, the dependent sum is empty:
\[
\sum_{c : \mathsf{Chord}} \mathsf{domfunc}(c, s) = \varnothing.
\]
When the domain of quantification is empty, the dependent product over it is the terminal type $\mathbf{1}$, representing a proposition that is trivially true. Thus, for any scale $s$ in which there are no dominant-function chords, the proposition 
\[
\text{``Any chord $c$ in $s$ that has dominant function contains the leading tone of $s$''} 
\]
is vacuously true and is represented by the type $\mathbf{1}$, the unique type with exactly one proof.
\end{example}

\section{$\mathsf{W}$-Types}
\label{sec:w-types}

In this appendix we introduce $\mathsf{W}$-types, which describe inductive types whose terms are well-founded trees.\footnote{See \cite{martinlof1979,martin2021intuitionistic,hott2013} for foundational treatments. For the categorical semantics of $\mathsf{W}$-types as initial algebras of polynomial endofunctors, see \cite{gambino2013polynomial,moerdijk2000wellfounded}. We omit this material here in the interest of space.} A $\mathsf{W}$-type is determined by two pieces of data: a type $A : \mathsf{Type}$ and a type family $B : A \to \mathsf{Type}$. Each element $a : A$ serves as a label for a node in the tree, while the type $B(a)$ specifies the arity of that label—that is, the quantity of branches emanating from a node labeled by $a$.

To specify the subtrees attached to a node labeled by $a$, we provide a function
\[
f : B(a) \longrightarrow \mathsf{W}_{x : A} B(x),
\]
which assigns to each branch $b : B(a)$ a subtree of type $\mathsf{W}_{x : A} B(x)$. Thus, a term of the $\mathsf{W}$-type consists of a label together with a family of subtrees indexed by its arity.

As a simple example, suppose we wish to define the type of lists $\mathsf{List}(A)$ whose elements are drawn from a type $A$. Such a list type can be presented as a $\mathsf{W}$-type with $\mathbf{1} + A$ labels: one nullary label corresponding to the empty list, and one label for each $a : A$ corresponding to adding $a$ to the head of a list.\footnote{See {\cite[155]{hott2013}.}}

To construct a list over $A$ consists in repeatedly choosing elements to append to the list, terminating after a finite sequence. In terms of the associated $\mathsf{W}$-type, a single step of this construction is represented by a dependent pair
\[
\sum_{u : \mathbf{1} + A} \bigl(B(u) \longrightarrow \mathsf{W}_{x : \mathbf{1} + A} B(x)\bigr),
\]
whose elements are pairs $(u, f)$ with $u : \mathbf{1} + A$ and $f$ a function assigning to each branch $b : B(u)$ a subtree (in this case, a sublist). Since $B(\mathsf{inl}(\star)) = \mathbf{0}$ and $B(\mathsf{inr}(a)) = \mathbf{1}$, this means:

\begin{itemize}
  \item If $u = \mathsf{inl}(\star)$, there are no branches, and hence no elements to append; the list terminates here.
  \item If $u = \mathsf{inr}(a)$, there is a single branch $\star : B(u)$, and $f(\star)$ is the tail of the list.
\end{itemize}

A term of the list type is obtained by applying the constructor
\[
\mathsf{sup} : \left( \sum_{u : \mathbf{1} + A} \left(B(u) \longrightarrow \mathsf{W}_{x : \mathbf{1} + A} B(x)\right) \right) \longrightarrow \mathsf{W}_{x : \mathbf{1} + A} B(x).
\]

For example:
\begin{itemize}
  \item The pair $(\mathsf{inl}(\star), !)$, where $! : \mathbf{0} \to \mathsf{W}_{x : \mathbf{1} + A} B(x)$ is the unique function, gives
  \[
  \mathsf{sup}(\mathsf{inl}(\star), !) : \mathsf{W}_{x : \mathbf{1} + A} B(x),
  \]
  which represents the empty list.
  
  \item A pair $(\mathsf{inr}(a), f)$, where $f : \mathbf{1} \to \mathsf{W}_{x : \mathbf{1} + A} B(x)$, gives a list whose head is $a$ and whose tail is $f(\star)$. If $f(\star)$ is itself of the form $\mathsf{sup}(\mathsf{inr}(b), f_b)$, then the resulting list has head $a$ and next element $b$, with tail $f_b(\star)$, and so on.
\end{itemize}

Unfolding this inductive construction produces a finite sequence of nested $\mathsf{sup}$-applications that eventually terminates in the empty list. For example,
\[
\mathsf{sup}\!\left(
  \mathsf{inr}(a), \lambda\star.\,
    \mathsf{sup}\!\left(
      \mathsf{inr}(b), \lambda\star.\,
        \mathsf{sup}\!\left(
          \mathsf{inr}(c), \lambda\star.\,
            \mathsf{sup}(\mathsf{inl}(\star), !)
        \right)
    \right)
\right)
\]
is a term of $\mathsf{W}_{x : \mathbf{1} + A} B(x)$ in which each $\lambda$-abstraction denotes the required function of type $\mathbf{1} \to \mathsf{W}_{x : \mathbf{1} + A} B(x)$. This term represents the list whose head is $a$, whose next element is $b$, whose next element is $c$, and whose tail is the empty list: in other words, the list $[a, b, c]$. This inductive structure can also be visualized as a tree, with each successive element corresponding to a subtree (see Figure~\ref{fig:inductive-tree}).

\begin{figure}[ht]
  \centering
  \[
  \begin{tikzcd}
    {\mathsf{sup}\!\left(   \mathsf{inr}(a),\    \lambda\star.\,     \mathsf{sup}\!\left(       \mathsf{inr}(b),\        \lambda\star.\,         \mathsf{sup}\!\left(           \mathsf{inr}(c),\            \lambda\star.\,             \mathsf{sup}\!\left(               \mathsf{inl}(\star),\                !             \right)         \right)     \right) \right)} &&& a \\
    {\mathsf{sup}\!\left(   \mathsf{inr}(b),\    \lambda\star.\,     \mathsf{sup}\!\left(       \mathsf{inr}(c),\        \lambda\star.\,         \mathsf{sup}\!\left(           \mathsf{inl}(\star),\            !         \right)     \right) \right)} &&& b \\
    {\mathsf{sup}\!\left(   \mathsf{inr}(c),\    \lambda\star.\,     \mathsf{sup}\!\left(       \mathsf{inl}(\star),\        !     \right) \right)} &&& c \\
    {\mathsf{sup}(\mathsf{inl}(\star), !)}
    \arrow[equals, from=1-1, to=1-4]
    \arrow[from=1-1, to=2-1]
    \arrow[from=1-4, to=2-4]
    \arrow[equals, from=2-1, to=2-4]
    \arrow[from=2-1, to=3-1]
    \arrow[from=2-4, to=3-4]
    \arrow[equals, from=3-1, to=3-4]
    \arrow[from=3-1, to=4-1]
  \end{tikzcd}
  \]
  \caption{The inductive structure of the list $[a, b, c]$ represented as nested $\mathsf{sup}$-applications, where each subtree corresponds to the tail of the list above it.}
  \label{fig:inductive-tree}
\end{figure}
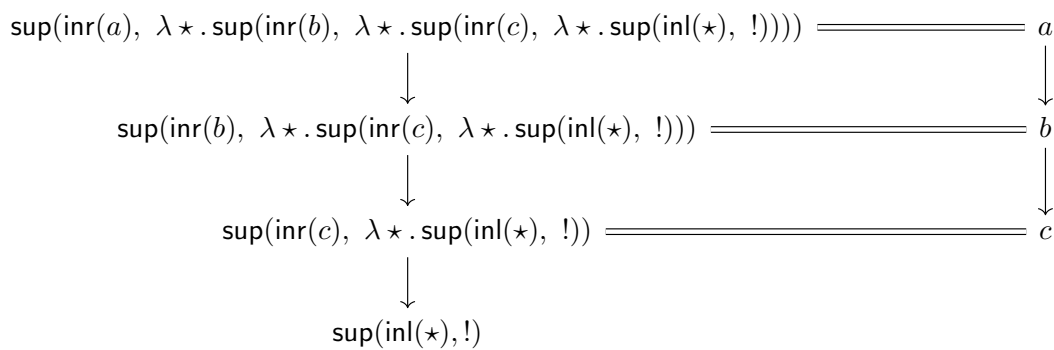

\begin{example}[Rhythm trees]
\label{ex:rhythm-trees}
A \emph{rhythm tree} is a rooted tree structure in which each node is labeled by a rational number representing a duration. The branches of a node correspond to proportional factors that subdivide that duration. These proportional values need not sum to the duration of the parent node: they simply express relative subdivisions. For example, a duration of $5$ can be divided proportionally into $1, 1, 1$, which indicates three equal parts.

We now show how this structure can be encoded as a $\mathsf{W}$-type. To do so, we  specify a type $A$ of node labels and a type family $B : A \to \mathsf{Type}$ describing the arities of those nodes.

Let $\mathbb{Q}$ denote the type of rational numbers. Any rational $d : \mathbb{Q}$ may be paired with a finite list of rational factors representing its proportional branches. We therefore define
\[
A \coloneqq \mathbb{Q} \times \mathsf{List}(\mathbb{Q}),
\]
whose elements are pairs $(d, [r_1, \ldots, r_n])$, where $[r_1, \ldots, r_n]$ is a list of $n$ rational proportional factors associated with $d$. These factors specify how $d$ is divided proportionally.

Next, we define the arity family
\[
\begin{matrix}
B : & A & \longrightarrow & \mathsf{Type} \\
& (d, [r_1, \ldots, r_n]) & \longmapsto & \mathsf{Fin}(n)
\end{matrix}
\]
which maps each pair $(d, [r_1, \ldots, r_n])$ to the finite type $\mathsf{Fin}(n)$ of positions $0, \dots, n-1$. Each element $k : \mathsf{Fin}(n)$ corresponds to the subtree attached to the branch labeled $r_k$.

A term of the resulting $\mathsf{W}$-type is then constructed by
\[
\mathsf{sup}((d, [r_1, \ldots, r_n]), f)
\]
where
\[
f : \mathsf{Fin}(n) \longrightarrow \mathsf{W}_{x : A} B(x)
\]
assigns to each branch position $k : \mathsf{Fin}(n)$ a subtree $f(k)$, itself built inductively in the same manner. In particular, each subtree takes the form
\[
f(k) = \mathsf{sup}((d', [s_1, \ldots, s_m]), g)
\]
for some $d' : \mathbb{Q}$, list $[s_1, \dots, s_m] : \mathsf{List}(\mathbb{Q})$ of proportional factors, and function
\[
g : \mathsf{Fin}(m) \longrightarrow \mathsf{W}_{x : A} B(x).
\]

Thus a rhythm tree is built inductively by repeatedly choosing a rational duration, specifying a finite list of proportional factors, and attaching recursively constructed rhythm trees to each factor (see Figure~\ref{fig:rhythm-tree}).
\begin{figure}[ht]
  \centering
  \[
  \begin{tikzcd}[column sep=1.5em, row sep=2em]
    &&& 9.5 \\
    2 &&& 2.5 &&& 3 \\
    && 1 & 1 & 1 & 1.5 && 2 \\
    & 2 & 1
    \arrow[from=1-4, to=2-1]
    \arrow[from=1-4, to=2-4]
    \arrow[from=1-4, to=2-7]
    \arrow[from=2-4, to=3-3]
    \arrow[from=2-4, to=3-4]
    \arrow[from=2-4, to=3-5]
    \arrow[from=2-7, to=3-6]
    \arrow[from=2-7, to=3-8]
    \arrow[from=3-3, to=4-2]
    \arrow[from=3-3, to=4-3]
  \end{tikzcd}
  \]
  \caption{A rhythm tree with root duration $9.5$ subdivided into proportional branches. Each branching node represents a rational proportional factor, which may itself branch further.}
  \label{fig:rhythm-tree}
\end{figure}
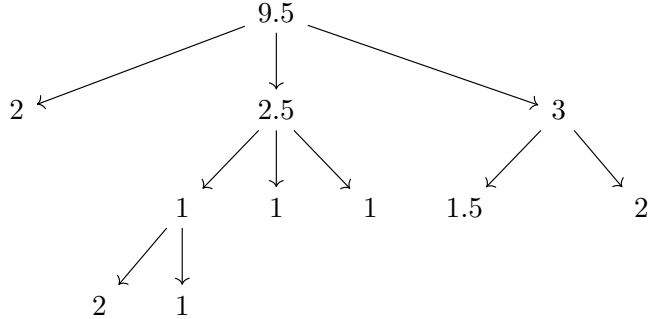
\end{example}

\end{document}